\documentclass[12pt]{amsart}
\usepackage{amssymb, mathrsfs, comment, xcolor}
\usepackage[showonlyrefs]{mathtools}
\usepackage[colorlinks=true, allcolors=black]{hyperref}
\usepackage[a4paper, hmargin=3.3cm, vmargin={3.7cm, 3.7cm}]{geometry}
\usepackage{titlesec}
\numberwithin{equation}{section}
\newcommand{\periodafter}[2]{#1{#2.}}
\newcommand*{\justifyheading}{\raggedright}
\titleformat {\section}
{\bfseries\justifyheading\MakeUppercase}
{\thesection}{2em}{\periodafter{}}
\titleformat {\subsection}
{\bfseries\justifyheading}{\thesubsection}{1em}{\periodafter{}}

\usepackage{esint}
\usepackage{graphicx}
\usepackage{bm}
\usepackage{comment}
\usepackage{epsfig}

\usepackage{mathrsfs}
\usepackage{enumitem}
\usepackage{color}
\usepackage{amsfonts}
\usepackage{verbatim}
\usepackage{amsthm}
\usepackage{dsfont}
\usepackage{url}
\usepackage{esint}
\usepackage{tocloft}
\usepackage{tikz}
 \usetikzlibrary{arrows}    
 \usetikzlibrary{calc}  
\usetikzlibrary{decorations.markings}
\usetikzlibrary{arrows.meta}
\usepackage[mathcal]{euscript}
\DeclareMathOperator{\R}{Re}
\DeclareMathOperator{\I}{Im}
\DeclareMathOperator{\diam}{diam}

\DeclareMathOperator*{\wsc}{\overset{*}{\rightarrow}}

\theoremstyle{definition}
\newtheorem{definition}{Definition}[section]

\theoremstyle{remark}
\newtheorem{remark}[definition]{Remark}

\theoremstyle{plain}
\newtheorem{theorem}[definition]{Theorem}
\newtheorem{lemma}[definition]{Lemma}
\newtheorem{proposition}[definition]{Proposition}
\newtheorem{corollary}[definition]{Corollary}

\newcommand{\caliA}{\mathcal{A}}
\newcommand{\caliD}{\mathcal{D}}
\newcommand{\caliE}{\mathcal{E}}

\newcommand{\caliU}{\mathcal{U}}

\newcommand{\cj}{\overline}
\newcommand{\D}{\mathbb{D}}
\newcommand{\T}{\mathbb{T}}
\newcommand{\N}{\mathbb{N}}
\newcommand{\C}{\mathbb{C}}

\newcommand{\vfi}{\varphi}
\newcommand{\vep}{\varepsilon}

\newcommand{\diff}{{\rm d}}
\newcommand{\diffA}{{\rm dA}}
\newcommand{\Ordo}{\mathrm{O}}

\numberwithin{equation}{section}
\begin{document}

\title[On Bergman polynomials for domains with corners]
{Asymptotics of Bergman polynomials for domains with reflection-invariant corners}
\author{Erwin Mi\~{n}a-D\'{\i}az \and Aron Wennman}
\dedicatory{Dedicated to Arno B.J. Kuijlaars on the occasion 
of his 60\hspace{1pt}\textsuperscript{th} birthday}
\date{\today}


\begin{abstract}
We study the asymptotic behavior of the Bergman orthogonal polynomials $(p_n)_{n=0}^{\infty}$ 
for a class of bounded simply connected domains $D$. The class is defined by the
requirement that conformal maps $\vfi$ of $D$ onto the unit 
disk extend analytically across the boundary $L$ of $D$, 
and that $\vfi'$ has a finite number of zeros $z_1,\ldots,z_q$ on $L$. 
The boundary $L$ is then piecewise analytic with corners at the zeros of $\vfi'$. 
A result of Stylianopoulos implies that
a Carleman-type strong asymptotic formula 
for $p_n$ holds on the exterior domain $\C\setminus\overline{D}$.
We prove that the same formula remains
valid across $L\setminus\{z_1,\ldots,z_q\}$ and 
on a maximal open subset of $D$. As a consequence, 
the only boundary points that attract zeros of $p_n$ are the corners.
This is in stark contrast to the case when $\vfi$ fails to admit an analytic
extension past $L$, 
since when this happens the zero counting measure of $p_n$ is known to approach the 
equilibrium measure for $L$ along suitable subsequences.
\end{abstract}

\maketitle

\section{Introduction}
\subsection{Zeros of Bergman polynomials for domains with corners}
\noindent 
For a domain $D$ in the complex plane $\C$, we write $A^2(D)$ for the Bergman space of 
holomorphic functions $f$ on $D$ with finite norm
\[
\|f\|_{L^2(D)}:=\left(\int_D|f(z)|^2\diffA(z)\right)^{1/2}<\infty,
\]
where $\diffA(z)=\pi^{-1}\diff x \diff y$ denotes the standard area measure normalized by $\pi$.
In this work, $D$ will always stand for the interior domain of a Jordan curve $L$. 
That is, $D$ is the bounded component of $\C\setminus L$.  
We will denote by $\D(z,r)$ the open disk of radius $r$ centered at $z$, by
$\T(z,r)$ its boundary, and by $\Delta(z,r)$ the exterior disk 
$\Delta(z,r)=\cj{\C}\setminus\overline{\D(z,r)}$.

The space of polynomials forms a linear subspace of $A^2(D)$. 
By applying the Gram-Schmidt orthonormalization procedure to the standard monomial sequence, 
we obtain a unique sequence of polynomials $(p_n)_{n=0}^\infty$ such that $p_n$ has 
degree $n$ and positive leading coefficient, and the following planar orthonormality 
conditions are satisfied:
\begin{align*}
\int_Dp_n(z)\cj{p_m(z)}\,\diffA(z)=\begin{cases}0, & \ m\not=n,\\
1, & \ n=m.
\end{cases}
\end{align*}
These polynomials are often called the Bergman polynomials for the domain $D$.
Their study dates back at least to Carleman's 1922 paper \cite{Carleman},
and the name Carleman polynomials also occurs in the literature.

For a positive integer $N\ge 3$, we denote by $R_N$ the regular $N$-gon with vertices
at the $N$-th roots of unity. The zeros of the associated Bergman polynomials 
are known to exhibit a curious dichotomy depending on the value of $N$.
In numerical experiments carried out in \cite{ES}, it was observed that for $N\in\{3,4\}$,
the zeros of $p_n$ appear to lie on the set $\Gamma_N$ formed by 
joining the $N$ vertices of $R_N$ to its center, 
and indeed this observation was later confirmed to be a fact in \cite{MaymeskulS}. 
However, for $D=R_N$ with $N\geq 5$, the experiments in \cite{ES} strongly indicate that the zeros of 
$p_n$ move away from $\Gamma_N$ as $n$ increases. 
The reason behind this dichotomy was partially explained in the work \cite{LevinSS}, 
and as we describe in what follows, it is related to the analytic continuation properties 
of the interior conformal maps of $D$ onto the unit disk. 

Denote by $\nu_n$, $n\geq 1$, the normalized counting measure for the zeros set of $p_n$. 
That is, if $\zeta_{1,n},\ldots,\zeta_{n,n}$ denote the zeros of $p_n$ (counted with multiplicity),
we put
\[
\nu_{n}:=\frac{1}{n}\sum_{j=1}^n\delta_{\zeta_{j,n}}.
\]
For a measure $\nu$, we write $\nu_n\wsc \nu$ as $n\to\infty$ to mean that 
\[
\lim_{n\to\infty}\int f\diff\nu_n=\int f \diff\nu
\] 
for every compactly supported continuous function $f$ on $\C$.
We denote by $Z$ the set of points $z_0\in \C$ with the property that for every open neighborhood $U$ 
of $z_0$, one can find infinitely many polynomials in the sequence $(p_n)$ that have a zero in $U$. 
It is easy to see that if $\nu_n\wsc \nu$, then every point in the support of $\nu$ belongs to $Z$.
The following result was proven in \cite{LevinSS} (see Theorem 2.1 and Fact B therein).

\begin{theorem}[Levin, Saff, Stylianopoulos \cite{LevinSS}]\label{LSSThm} 
Let $D$ be the interior domain of a Jordan curve $L$, 
and let $\sigma_L$ be the equilibrium measure of $L$. 
The following statements are equivalent: 
\begin{enumerate}[leftmargin=.9cm]
\item[{\rm (i)}]  No conformal map of $D$ onto $\D(0,1)$ can 
be analytically continued to a domain containing $\cj{D}$.  
\item[{\rm (ii)}] There exists a subsequence $(\nu_{n_k})$ with 
\begin{align}\label{wc}
\nu_{n_k}\wsc \sigma_L\quad  \mathrm{as}\quad k\to\infty.
\end{align}	
\item[{\rm (iii)}] For every $z\in D$,
\[
\limsup_{n\to\infty}|p_n(z)|^{1/n}=1.
\]
\end{enumerate}	
\end{theorem}

Note that {\rm (i)} is equivalent to the existence 
of \emph{one} conformal map $\varphi$ which cannot
be analytically continued past $L$.

We remark that there are classes of domains for which the convergence in \eqref{wc} is 
known to hold for the full sequence $(\nu_n)$, see \cite[Cor. 3.1]{SaffStylianopoulos}.

Since the support of $\sigma_L$ is all of $L$, the convergence \eqref{wc} implies that if 
$\varphi$ has a singularity on $L$, then $L\subset Z$. When $D=R_N$ with $N\geq 5$, the 
conformal map $\varphi$ has singularities at the vertices of $D$, so that every point 
of $L$ attracts zeros of the polynomials $p_n$. As a consequence, it is not possible 
for the zeros to stay on the set $\Gamma_N$, as it otherwise happens when $D$ is the 
equilateral triangle ($D=R_3$) or the square ($D=R_4$). We note that when $N\in\{3,4\}$, 
not only are the zeros of every $p_n$ located on the set $\Gamma_N$, but we also
have that $\Gamma_N\subset Z$. This  is a consequence of Theorem 9 of \cite{MaymeskulS}.    

\begin{definition}
We let $\caliA_1$ denote the collection of bounded Jordan domains $D$ 
with the property that a conformal map $\vfi$ of $D$ onto $\D(0,1)$ has an analytic 
continuation to an open set containing $\cj{D}$, and the derivative $\vfi'$ has 
at least one zero on the boundary $L$ of $D$. 
\end{definition}

\begin{figure}
	\begin{center}
\scalebox{1}{\begin{tikzpicture}
			\draw (-3.85, 0) node[inner sep=0] {\includegraphics[scale=.8]{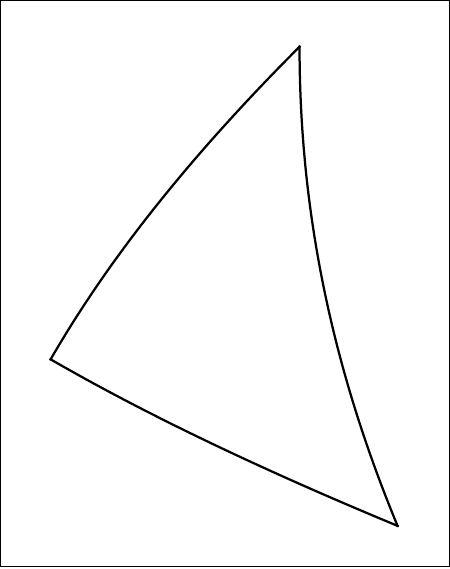}};
		\draw (3.85, 0) node[inner sep=0] {\includegraphics[scale=.8]{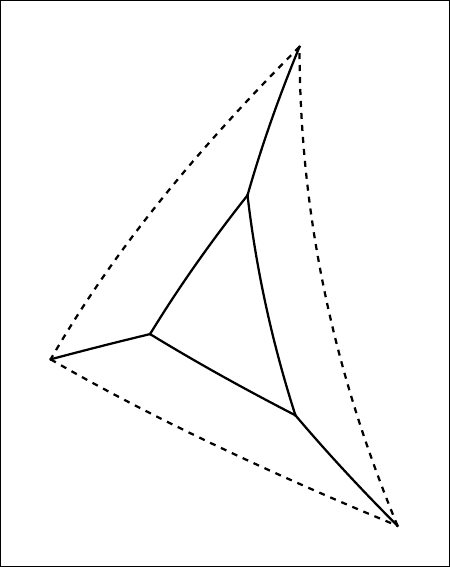}};
		\draw (-2.7, 3.45) node {$z_1$};
	\draw (-6.53, -1.07) node {$z_2$};
	\draw (-1.23, -3.4) node {$z_3$};
	\draw (-6.2,3.2) node {$\Omega$};
	\draw (-1.5,1) node {$L$};
	\draw (-4,-.3) node {$D$};
	 \draw[-Stealth]        (-1.7,1)   -- (-2.68,1);
	 
	 	\draw (-2.7+7.7, 3.45) node {$z_1$};
	 \draw (-6.53+7.7, -1.07) node {$z_2$};
	 \draw (-1.23+7.7, -3.4) node {$z_3$};
	 \draw (-6.2+7.7,3.2) node {$\Omega^*$};
	 \draw (-1.5+7.7,1.5) node {$\partial \Omega^*$};
	 	 \draw[-Stealth]        (5.75,1.5)   -- (4.26,1.5);
	\end{tikzpicture}}
\end{center}
	\caption{Illustration of the curve $L=\partial D$ and sets $\Omega$ and $\Omega^*$, for a domain $D\in \caliA_1$ with three corners $z_1$, $z_2$, and $z_3$.}
	\label{curveL}
\end{figure}
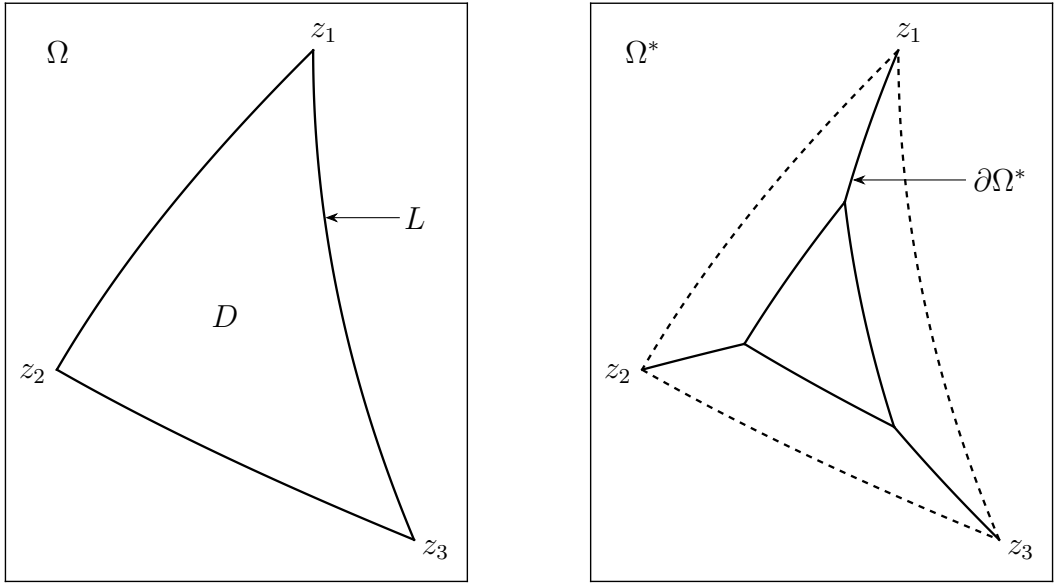

A domain $D\in\caliA_1$ is said to have \emph{reflection-invariant corners}.
For an explanation of this terminology, see the last paragraph of this section.

If $D\in\caliA_1$, then $L$ is a piecewise analytic curve with corners at each of the zeros 
of $\varphi'$ on $L$. We denote this set of corners by $C(L)=\{z_1,\ldots,z_q\}$ 
(see Figure \ref{curveL}). The interior 
angle formed by the two arcs of $L$ that meet at a given corner $z_j$ is equal to 
$\pi/m_j$, where $m_j\geq 2$ is the order of $\vfi$ at $z_j$. 
The equilateral triangle and the square both belong to the class $\caliA_1$.
We might also mention the example of 
the lens-shaped domain $D$ bounded by two circular arcs that are symmetric about the 
imaginary axis and that meet at $-i$ and $i$, forming two corners with interior angles $\pi/2$. 
This domain is also in $\caliA_1$, and it was proven in \cite[Cor. 3.2]{LevinSS} that the 
zeros of each $p_n$ lie on the segment $(-i,i)$, and that $[-i,i]\subset Z$.

These examples suggest the hypothesis that for every $D\in \caliA_1$, the set $L\cap Z$ 
contains every corner of $L$ but no other boundary point (i.e. $L\cap Z=C(L)$). 
That this is indeed the case will follow from the main result of this paper, namely,
that the strong asymptotic formula known to describe the asymptotic behavior of $p_n$ as 
$n\to \infty$ on the exterior of $L$ in fact extends across the analytic arcs of 
$L$ to a portion of the domain $D$. We state this result in more precise terms in what follows.

\subsection{Strong asymptotics on a maximal open set}
\noindent Denote by $\Omega$ the unbounded component of $\cj{\C}\setminus L$, 
so that $\Omega=\cj{\C}\setminus \cj{D}$ (see Figure \ref{curveL}, left side). 
We let $\phi:\Omega\to\Delta(0,1)$ be the conformal map of 
$\Omega$ onto the exterior disk $\Delta(0,1)$, such that $\phi(\infty)=\infty$ and
$\gamma:=\phi'(\infty)>0$. 
By Carath\'{e}odory's theorem, the map $\phi$ extends to a 
homeomorphism of $\cj{\Omega}$ onto $\cj{\Delta(0,1)}$.

In \cite{Nikos1}, Stylianopoulos proved that if the orthogonality domain $D$ 
is bounded by a piecewise analytic curve without cusps, then $p_n$ satisfies
a Carleman-type strong asymptotic formula on the exterior domain $\Omega$.
The error term was later improved in \cite{BeckermannStylianopoulos}, and the
(sharp) asymptotic formula reads
\begin{align}\label{exteriorasymptotics}
p_n(z)=\sqrt{n+1}\phi'(z)\phi(z)^n\big(1+\Ordo(n^{-1})\big),\qquad z\in \Omega,
\end{align}
as $n\to\infty$, with the implicit constant being uniformly bounded on closed subsets of $\Omega$.
In particular, \eqref{exteriorasymptotics} applies whenever $D\in \caliA_1$. 
Note that if the piecewise analytic curve $L$ is such that an 
interior conformal map $\varphi: D\to\D(0,1)$ has a singularity on $L$ 
(for instance, if the interior angle $\alpha$ formed at a corner is such that $\pi/\alpha\not\in \N$), 
then the asymptotic formula \eqref{exteriorasymptotics} cannot possibly hold on an open set 
larger than $\Omega$. Indeed, by Theorem \ref{LSSThm} every point of $L$ attracts zeros of the 
polynomials $p_n$, and furthermore, the $\limsup_{n\to\infty}\sqrt[n]{|p_n(z)|}$ 
remains constant (equal to $1$) throughout $D$.
However, the opposite turns out to be true when $D\in \caliA_1$.
This is a direct consequence of the following theorem, which is the main
result of this paper.

\begin{theorem}\label{thm:exteriorPhi}
Assume $D\in \caliA_1$. There exists an open set 
$\Omega^*\supset \cj{\Omega}\setminus C(L)$ 
and a univalent function $\Phi:\Omega^*\to \cj{\C}$ with $\Phi|_\Omega=\phi$,
such that the asymptotic formula
\begin{align}\label{exteriorasymptoticsPhi}
p_n(z)=\sqrt{n+1}\Phi'(z)\Phi(z)^n\left(1+\Ordo\left(\frac{\log n}{n}\right)\right),\qquad z\in \Omega^*,
\end{align}
holds uniformly on closed subsets of $\Omega^*$ as $n\to\infty$. 
\end{theorem}
 
The set $\Omega^*$ is illustrated in Figure \ref{curveL} (right side), and will be precisely defined in Section \ref{s:gluing-overview}. This set is maximal in the sense 
that no formula such as \eqref{exteriorasymptoticsPhi} can 
hold on an open set larger than $\Omega^*$, since 
$\limsup_{n\to\infty}\sqrt[n]{|p_n(z)|}$ 
will be shown to remain constant on the interior of $\C\setminus\Omega^*$.    

In the examples where $D=R_N$ with $N\in\{3,4\}$, we have $\Omega^*=\cj{\C}\setminus \Gamma_N$ and $\Phi$ 
is the analytic continuation to $\Omega^*$ of the exterior conformal map $\phi$. However, 
we do not know whether $\Omega^*$ is always a connected set, and so we can only say that $\Phi$ 
coincides with the analytic continuation of $\phi$  on  the unbounded component of $\Omega^*$. 

Prior to \cite{Nikos1}, asymptotic formulas like \eqref{exteriorasymptotics} (but with different
estimates of the error terms) 
had been established  for  orthogonality domains $D$ that are bounded either by  an analytic Jordan 
curve \cite{Carleman},  or by a curve with some degree of smoothness \cite{Suetin}. 
For the class $\caliA_2$ of domains with analytic boundary, Carleman found already in 1922 
\cite{Carleman} that \eqref{exteriorasymptotics} indeed holds on the largest domain 
$\Omega_\rho \supset \cj{\Omega}$ that admits a univalent continuation of $\phi$ with 
$\phi(\Omega_\rho)=\Delta(0,\rho)$, $0\leq \rho<1$. More recently in \cite{DragnevMinaDiaz1} (see also \cite{DragnevMinaDiaz2,DragnevMinaNorthington,MinaDiaz2}), 
it was found that when $L$ is analytic, it is possible to characterize the largest open set 
$\Omega^*\supset \Omega_\rho$ on which a formula of the form 
\eqref{exteriorasymptoticsPhi} is valid.

Notice that a domain $D$ is in the union $\caliA=\caliA_1\cup\caliA_2$ precisely when a 
conformal map of $D$ onto $\D(0,1)$ is analytic in $\cj{D}$ 
(i.e., if the first statement of Theorem \ref{LSSThm} does not hold). 
For $D\in \caliA$, it turns out that the dominant behavior of $p_n(z)$ on $D$ 
as $n\to\infty$ is given by the integral formula
\begin{align}\label{integralformula}
p_n(z)\sim{} &	\frac{\sqrt{n+1}\varphi'(z)}{2\pi i}
\int_{|w|=1}\frac{w^n}{h(w)-\varphi(z)}\,\diff w,\qquad z\in D,
\end{align}  
where $h$ is the homeomorphism $h:\T(0,1)\to\T(0,1)$ given by
\begin{align*}
h(w):=\varphi(\psi(w))
\end{align*}	
and $\psi$ is the inverse of the exterior map $\phi$.
Since the right-hand side of \eqref{integralformula} makes sense for every Jordan domain, it may well be that 
this representation of $p_n$ is valid also for other domains outside the class $\caliA$. 
However, there is one feature which accounts for the distinguished behavior of the Bergman polynomials
for domains in the class $\caliA$. 
Namely, the fact that $D\in \caliA$ if and only if the function $h(w)$ appearing in the denominator of 
\eqref{integralformula} is analytic in a neighborhood of the unit circle $\T(0,1)$. 
This is the key feature that makes possible to extend the asymptotic formula 
\eqref{exteriorasymptotics} beyond $\Omega$ to a maximal subset of the orthogonality domain.

Theorem \ref{thm:exteriorPhi} has several implications for the limiting zero distribution 
of the Bergman polynomial $p_n$ for a domain $D\in\caliA_1$. 
The most immediate is that for any closed subset $E$ of $\Omega^*$, 
there is an index $n_E$ such that no polynomial $p_n$ of degree $n>n_E$ has a zero on $E$. 
As a consequence, the zero limit set $Z$ is contained in $\C\setminus \Omega^*$, 
and the only points of $L$ that attract zeros of the $p_n$'s are the corners of $L$. 

The fact that every corner point of $L$ is in $Z$ can easily be 
seen by applying standard arguments from potential theory. 
Indeed, the support $\mathrm{supp}(\nu)$ of any measure $\nu$ 
that is a weak $*$-limit of the sequence of zero counting measures 
$(\nu_n)$ is contained in $Z$, and  $\nu$ is an \emph{inverse Balayage measure} 
of the equilibrium measure $\sigma_L$ of $L$. This means that the logarithmic 
potential $U^{\nu}(z)$ of $\nu$ equals the logarithmic potential 
$U^{\sigma_L}(z)$ of $\sigma_L$ for all $z\in \Omega$. 
The potential $U^\nu$ is harmonic outside $\mathrm{supp}(\nu)$, 
and $U^{\sigma_L}(z)=\log\phi'(\infty)-\log|\phi(z)|$ for $z\in \Omega$. 
Since $\log|\phi(z)|$ cannot be harmonically extended past a corner of $L$, 
every corner has to be an element of $\mathrm{supp}(\nu)$.

In a forthcoming work, we will carry out a more comprehensive analysis 
on the distribution of the zeros of $p_n$ and on the structure of $\C\setminus \Omega^*$. 
Some of the results are similar to those obtained in \cite{DragnevMinaDiaz1} 
for domains with analytic boundary, but others are more specific to the 
class $\caliA_1$ and a manifestation of the presence of corners on $\partial D$.

To conclude this section, we briefly mention that
the class $\caliA_1$ can be characterized purely in geometric terms. 
Suppose that $D$ is a domain bounded by a piecewise analytic curve $L$, 
and that for every corner $z$ of $L$, two small sub-arcs $L_z^{0}$ and $L_z^{1}$ of 
$L$ that meet at $z$ form an interior angle of the form $\pi/m$, with $m\geq 2$ an integer. 
Applying a Schwarz reflection across $L_z^{1}$, we obtain a local holomorphic extension
of $\varphi$ to a wedge-shaped region between $L_{z}^{0}$ and a new arc $L_{z}^{2}$ 
(the conformal reflection of $L_z^{0}$ about $L_z^{1}$).
Iterating this, reflecting each time $L_{z}^{j}$ about $L_{z}^{j+1}$, 
we obtain after $2m-1$ steps an arc $L_z^{2m}$ which is tangent to $L_z^{0}$.
The domain $D$ belongs to $\caliA_1$ if and only if 
each point of $L_{z}^{0}$ is mapped to
itself by this $(2m-1)$-step reflection procedure, 
so that in particular $L_{z}^{0}=L_z^{2m}$. This is the
\emph{reflection invariance} alluded to in the title.

\section{Strong asymptotics across \texorpdfstring{$L$}{}}
\label{s:gluing-overview}
\subsection{Asymptotic behavior of \texorpdfstring{$p_n$ on $\Omega$}{}}
\noindent Let $A_n(z)$ be the analytic function on $\Omega$ defined implicitly by the
relation
\begin{align*}
p_n(z)=\sqrt{n+1}\phi'(z)\phi(z)^n\big(1+A_n(z)\big),\qquad z\in \Omega.
\end{align*}
As mentioned in the introduction, it was proven in \cite{Nikos1} that if $D$ is 
bounded by a piecewise analytic curve without cusps, 
then $\lim_{n\to\infty }A_n(z)= 0$ and the convergence is uniform 
on closed subsets of $\Omega$.
This is a consequence of the inequality (1.9) of \cite{Nikos1}, which asserts
that for some constant $C$,
\[
|A_n(z)|\leq \frac{C}{\sqrt{n}(|\phi(z)|-1)}+\frac{C}{n},\qquad  z\in\Omega,\quad n\ge 1.
\]
We will make use of a different inequality which provides a better estimate on the rate of decay of $A_n(z)$ for $z$ in $\Omega$ and away from the corners of $L$. We use ${\rm d}_{\C}(E, F)$ to denote the Euclidean distance between two closed sets $E$ and $F$, and when $E=\{z\}$ is a singleton, we simply write ${\rm d}_{\C}(z, F)$.
\begin{proposition}\label{PROP:BECKSTYL}
Suppose that the boundary $L$ of $D$ is a piecewise analytic curve without cusps whose set of corners is denoted by $C(L)$. For every $\epsilon>0$, there exists a constant $c_\epsilon$ such that 
\begin{align}\label{second-ineq-A_n}
|A_n(z)|\leq \frac{c_\epsilon|\phi(z)|}{n(|\phi(z)|-1)},\qquad z\in \caliE_\epsilon,\quad n>1,
\end{align}
where $\caliE_\epsilon:=\{z\in \Omega: {\rm d}_{\C}(z,C(L))>\epsilon\}$. In particular, $A_n(z)=\Ordo(1/n)$ uniformly on closed subsets of $\Omega$ as $n\to\infty$.

\end{proposition}

Though not explicitly stated, the inequality \eqref{second-ineq-A_n} is to a large extent contained in the proof of 
Theorem 1.1 in \cite{BeckermannStylianopoulos}. We will briefly go over the details at the start of Section \ref{s:proofs}.

\subsection{Asymptotic behavior of \texorpdfstring{$p_n$ on $D$}{}}\label{p_n-on-D}
\noindent Among the conformal maps of $D$ onto the unit disk $\D(0,1)$, 
we fix any one of them and call it $\varphi$. 

If $D\in\caliA_1$, then  $L$ is a piecewise analytic curve with corners at the zeros of  $\varphi'$ on $L$. This set of corners is denoted by $C(L)$. In this case, the map $\phi$ (and thus its derivative $\phi'$ as well) extends analytically to 
an open set containing $\cj{\Omega}\setminus C(L)$. The derivative $\phi'$ is integrable over $L$, since 
\[
\int_L|\phi'(z)|\,|\diff z|=\int_{\T(0,1)}|\diff w|=2\pi.
\]
This allows us to define the functions 
\begin{align}\label{definitionQ_n}
Q_n(z):=	\frac{(n+1)\varphi'(z)}{2\pi i}\int_{L}
\frac{\phi'(\zeta)\phi(\zeta)^n}{\varphi(\zeta)-\varphi(z)}\,\diff\zeta, 
\qquad z\in D,\quad n\geq 0.
\end{align}

We denote by $\lambda_n$ the leading coefficient of $p_n$, which is a positive number.
Recall also the notation $\gamma:=\phi'(\infty)$. The first step towards extending
the exterior strong asymptotics is the following series representation of $p_n$ on $D$.

\begin{theorem}\label{theorem:seriesexpansion}The Bergman polynomials $p_n$ for a domain $D\in \caliA_1$ admit the series representation 
\begin{align}\label{polythirdexpansion}
\frac{\lambda_n}{\gamma^{n+1}}p_n(z)={} &\sum_{j=0}^\infty h(n,j) Q_{n+j}(z), \qquad z\in D,\quad n \geq 0,
\end{align}
where the coefficients $h(n,j)$ are numbers such that $h(n,0)=1$ and 	\begin{align}\label{bounds-for-hn}
|h(n,j)|\leq \frac{B}{n+j+1}\left(1+\frac{j-1}{n+1}\right)^B ,\qquad j\geq 1,
\end{align}
for a certain  constant $B>0$ that is independent of $n$ and $j$.
\end{theorem}

Formula \eqref{polythirdexpansion} will be deduced from a well-known 
connection between the Bergman polynomials and the reproducing kernel of the space $A^2(D)$, 
that is, from the equality
\begin{align*}
\frac{\varphi'(z)\cj{\varphi'(\zeta)}}{(1-\varphi(z)\cj{\varphi(\zeta)})^2}
=\sum_{k=0}^\infty\cj{p_k(\zeta)}p_k(z),\qquad z,\zeta\in D,
\end{align*}
see Subsection \ref{proof-thm-series-expansion} for details. 
In doing so, we make use of some inequalities previously obtained in \cite{Nikos1}.
 
The dominant term of the series in \eqref{polythirdexpansion} happens to be the first one, 
and the behavior of the leading coefficient $\lambda_n$ is known from \cite{Nikos1}: 
\begin{align*}
\lambda_n=\sqrt{n+1}\gamma^{n+1}(1+\Ordo(1/n)), \qquad n\to\infty.
\end{align*}  
With proper control over the error terms, this  will lead us to conclude that for $z\in D$, 
$p_n(z)\sim (n+1)^{-1/2}Q_n(z)$ as $n\to\infty$, which is the same as the integral formula \eqref{integralformula}. 
This can be seen by making the change of variables $\zeta=\psi(w)$ in \eqref{definitionQ_n}. 

By analyzing the large $n$ behavior  of the integral in \eqref{integralformula}, 
we will extend the exterior asymptotics \eqref{exteriorasymptotics} to a certain maximal open set $D_1$ of $D$.  
The construction of $D_1$ is carried out in what follows and applies more generally to any domain $D\in \caliA$. Let us set
\[
L_r:=\{z\in \cj{\Omega}: |\phi(z)| =r\},\qquad r\in [1,\infty].
\]
Each $L_r$ with $1<r<\infty$ is an analytic Jordan curve contained in $\Omega$,  $L_1=L$, and $L_\infty=\{\infty\}$. 
We  denote by  $\caliD_r$  the component of $\cj{\C}\setminus L_r$ that does not contain $\infty$. 
Then, $\caliD_1=D$ and $\caliD_\infty=\C$.

Let $\mu\in[0,1]$ be the smallest positive number with the property that $\vfi$ extends meromorphically 
to $\caliD_{1/\mu}$. Since, by assumption, $\vfi$ extends analytically across $L$, we have $\mu<1$. 
Recall that
\[
\psi:\Delta(0,1)\to \Omega
\] 
denotes the inverse of $\phi$, which extends continuously to $\cj{\Delta(0,1)}$. The composite 
\[
h(w):=\vfi(\psi(w))
\] 
is well-defined and meromorphic in the annulus $1<|w|<1/\mu $, extending continuously 
to $1\leq |w|<1/\mu$ and mapping the unit circle $|w|=1$ onto itself. 
By the reflection principle (see e.g. \cite[Ch. 8]{Davis}), $h(w)$ can be extended as a meromorphic function to the 
bigger annulus $\mu<|w|<1/\mu$, the extension being given by 
\[
h(w):=\frac{1}{\cj{\vfi(\psi(1/\cj{w})})}, \qquad \mu<|w|<1.
\] 
Moreover, it is not difficult to see that $\mu<|w|<1/\mu$ is the largest annulus about 
the origin that supports a  meromorphic extension of $h$. 

We next split $D$ into subregions, determined by the zeros of the function 
\begin{align*}
f_z(w):=h(w)-\varphi(z)\,.
\end{align*}
First, we let $D_0$ denote the set of points $z\in D$ for which the function $f_z(w)$ has no zeros in the annulus $\mu<|w|<1$. 
Then, consider a point $z\in D\setminus D_0$. Among the zeros of $f_z(w)$ in the region $\mu<|w|<1$, only finitely many, 
say $w_{z,1}, w_{z,2}, \ldots, w_{z,s}$ ($s\geq 1$), will have largest modulus. 
Let $m_{z,k}$ denote the multiplicity of $f_z(w)$ at $w_{z,k}$. 
For every integer $p\geq 1$, we define $D_p$ as the set of points $z\in D\setminus D_0$ such that 
\begin{align*}
m_{z,1}+m_{z,2}+\cdots+ m_{z,s}=p.
\end{align*}
Thus, $D_p$ consists of those $z\in D$ for which $f_z(w)$ has exactly $p$ zeros in $\mu<|w|<1$ of largest modulus, 
counted according to multiplicity.  

Define next the function $r:D\to[\mu,1)$ by
\begin{equation}\label{eq5}
r(z):=\left\{\begin{array}{ll}
|w_{z,1}|, & z\in D\setminus D_0, \\
\mu,  & z\in D_0.
\end{array}\right.
\end{equation}
If $z\in D_1$, then   $f_z(w)$ has exactly one zero in $\mu<|w|<1$ 
of largest modulus, say $w_{z,1}$, and this zero is simple. Let us set 
\begin{align*}
\phi_1(z):=w_{z,1},\qquad z\in D_1,
\end{align*}
so that $|\phi_1(z)|>\mu$.

One can easily verify by using Lemma \ref{descriptionD_p} (see also Corollary 12 in 
\cite{DragnevMinaDiaz1}) that $D_1$  and $D\setminus D_0$ are open, 
that the map $\phi_1$ is a univalent function, and that the function $r(z)$ is continuous.
 
\begin{theorem}\label{theorem:strongasymptotic}
Assume $D\in \caliA_1$. 
The analytic function $A_n:D_1\to \D(0,1)$ implicitly defined by  
\begin{align}\label{asymptoticformulaindideD}
p_n(z)=\sqrt{n+1}\phi_1'(z)\phi_1(z)^n\left(1+A_n(z)\right), \qquad z\in D_1,
\end{align}
satisfies that $A_n(z)=\Ordo(n^{-1})$ as $n\to\infty$, uniformly on compact subsets of $D_1$. Moreover, 
\begin{align}\label{limsup-equality}
\limsup_{n\to\infty}|p_n(z)|^{1/n}=r(z),\qquad z\in D.
\end{align}
\end{theorem}

It will be important to provide a more
precise interior bound for $|A_n(z)|$  near $L$, similar to the exterior bound \eqref{second-ineq-A_n}. We will obtain an inequality of the form
\begin{align}\label{growth-at-boundary1}
\left|A_n(z)\right|\leq {} &\frac{M_1}{n(1-|\phi_1(z)|)}+\frac{M_1}{n^{B_1}(1-|\phi_1(z)|)^{B_1}},
\end{align}
valid for all  $z\in D$ sufficiently close to $L$ but away from the corners, 
see Section~\ref{section:asymptotics-near-boundary} below.
Here, $M_1$ and $B_1>1$ are some constants that do not depend on $n$.

\subsection{Gluing interior and exterior asymptotics}
\noindent As a consequence of Proposition \ref{prop-asympt-near-boundary} (the statement and proof of which we defer to
Section~\ref{section:asymptotics-near-boundary}), we have that the set 
\[
\Omega^*:=D_1\cup \cj{\Omega}\setminus C(L)
\]
is open, and the function $\Phi:\Omega^*\to \Delta(0,\mu)$ defined by 
\[
\Phi(z):=\begin{cases}
\phi(z), &\ z\in \cj{\Omega},\\
\phi_1(z), &\ z\in D_1,
\end{cases}
\]
is univalent. Hence  
\begin{align*}
A_n(z):=\frac{p_n(z)}{\sqrt{n+1}\Phi'(z)\Phi(z)^n}-1,\qquad z\in \Omega^*,
\end{align*}
is  analytic in $\Omega^*$. 
The main goal is to show that $A_n(z)\to 0$ on $\Omega^*$. 
To prove this, we start from the two
bounds \eqref{second-ineq-A_n} and \eqref{growth-at-boundary1}, which both degenerate as $z$ approaches $L$.
These bounds show that $A_n(z)$ is small, except for possibly near $L$.
More concretely, we have that $A_n(z)=\Ordo(n^{-1}\log n)$ on the portion of $\Omega^*$ outside the shrinking band
\[
\Big\{z\in\C: \mathrm{d}_{\C}(z,L)\le \frac{1}{\log n}\Big\}
\]
around $L$.
Applying a Phragm\'{e}n–Lindel\"{o}f-type argument in 
this narrow band but away from the corners, 
we are able to deduce the following key lemma, which immediately implies 
Theorem~\ref{thm:exteriorPhi}.

\begin{lemma}\label{maintheorem:strongasymptotic} 
Assume  $D\in \caliA_1$, and fix any $z_0\in L\setminus C(L)$. 
There exists a neighborhood $\mathcal{U}_{z_0}\subset\Omega^*$ of $z_0$ and a corresponding constant $C_{z_0}$ such that
\begin{align*}
|A_n(z)|\leq \frac{C_{z_0}\log n}{n}\qquad z\in \mathcal{U}_{z_0},\quad n>1.
\end{align*}
\end{lemma}

\begin{remark}
Our extension of the strong asymptotics across $L$ comes at the expense of a 
factor $\log n$ in the estimate of the rate of decay.
Without too much effort it is possible to improve Lemma~\ref{maintheorem:strongasymptotic} 
to give $A_n(z)=\Ordo(n^{-1}\log^\ell (n))$ for any fixed $\ell\ge 1$,
where $\log^\ell(n)$ denotes the iterated logarithm 
\[
\log^\ell(n):=\underbrace{\log\log\ldots\log n}_{\ell\text{ times}}.
\]
However, we will not present the details here, 
since most likely the loss in precision is artificial. 
\end{remark}

The asymptotic results presented in this section as well as Theorem~\ref{thm:exteriorPhi}
will be proven in Section~\ref{s:proofs}. The proofs will rely on several auxiliary estimates
that will be convenient to have established beforehand.

\section{Auxiliary estimates}

\noindent In this section, we make use of some key estimates obtained in \cite{Nikos1}, 
and so for the most part we try to follow the notation employed therein.  

Let $n\geq 0$ be an integer. The Faber polynomial $F_n$  associated to the map $\phi$ is 
the polynomial part  of the Laurent expansion of $\phi(z)^n$ at infinity.  
Thus, $F_n(z)=\gamma^nz^n+\cdots $ is a polynomial of degree $n$ whose leading coefficient is $\gamma^{n}$.
Let $E_n(z)$ be defined by the equation 
\begin{align*}
\phi(z)^n=F_n(z)+E_n(z), \qquad z\in \Omega,
\end{align*}
so that for every $r$ such that $\Delta(0,r)\subset \Omega$, we have the Laurent expansion
\begin{align}\label{LaurentexpansionforEn}
E_n(z)=\frac{c_{n,1}}{z}+\frac{c_{n,2}}{z^2}+\frac{c_{n,3}}{z^3}+\cdots,\qquad z\in \Delta(0,r).
\end{align}
Note that since the map $\phi$ is continuous on $\cj{\Omega}$, so is $E_n$.
If we differentiate the identity 
\begin{align*}
\phi(z)^{n+1}=F_{n+1}(z)+E_{n+1}(z)
\end{align*}
we get 
\begin{align}\label{defGnandHn}
\phi'(z)\phi(z)^n=G_n(z)+H_n(z),  \qquad z\in \Omega,
\end{align}
with 
\begin{align}\label{secondkind-Faber}
G_n(z):={} &\frac{F'_{n+1}(z)}{n+1}=\gamma^{n+1}z^n+\cdots,\qquad n\geq 0, \\[.5em]
H_n(z):={} &\frac{E'_{n+1}(z)}{n+1}=\frac{a_{2,n}}{z^2}+\frac{a_{3,n}}{z^3}+\cdots,\qquad n\geq 0.\nonumber
\end{align}
One can verify, see e.g. \cite[Lemma 2.1]{Nikos1}, that 
\begin{align*}
H_n\in A^2(\Omega),\qquad n\geq 0.
\end{align*}

\begin{lemma}
Suppose that $L$ is rectifiable. For integers $n\geq 0$ and $m\geq -1$, we have   
\begin{align}\label{ident3}
\int_L \phi^m(z)\phi'(z)\cj{E_n(z)}\,\diff z=0. 
\end{align}
\end{lemma}

This is one of three identities stated in Lemma 2.2 of \cite{Nikos1}. We quickly verify its validity.

By \eqref{LaurentexpansionforEn}, we have  $\lim_{z\to\infty}\phi(z)E_n(z)=\gamma c_{n,1}\in \C$, 
so that we can write $E_n(z)=\phi(z)^{-1}E^*_n(z)$, with $E^*_n(z)$ analytic in $\Omega$ 
and continuous on $\cj{\Omega}$.  Making the change of variables $z=\psi(w)$ 
(recall that $\psi$ is the inverse of $\phi$) and  using that $w=1/\cj{w}$ for $w$ on the unit circle, we get  
\begin{align*}
\int_L \phi(z)^m\phi'(z)\cj{E_n(z)}\,\diff z={} & \int_{|w|=1}w^{m+1}\cj{E^*_n(\psi(1/\cj{w}))}\,\diff w.
 \end{align*}
The function $\cj{E^*_n(\psi(1/\cj{w}))}$ is analytic in the unit disk $\D(0,1)$ and 
continuous on $\cj{\D(0,1)}$. It follows from Cauchy's theorem that the latter integral equals zero, completing the proof of  \eqref{ident3}. 

Let us now define the polynomial  
\begin{align*}
q_{n-1}(z):=G_{n}(z)-\frac{\gamma^{n+1}}{\lambda_n}p_n(z),\qquad n\geq 0.
\end{align*}
Note that $q_{-1}\equiv 0$ and that $q_{n-1}$ is a polynomial of degree $\leq n-1$ for all $n\geq 1$. 
We also define, for a rectifiable $L$, the quantities
\begin{align}\label{defbeta-n}
\beta_{k,n}:=-\frac{1}{2\pi i}\int_{L}q_{k-1}(z)\cj{E_{n+1}(z)}\,\diff z,\qquad k\geq n\geq 0,
\end{align}
and
\begin{align}\label{defvarepsilon-kn}
\vep_{k,n}:=-\frac{1}{2\pi i}\int_{L}H_k(z)\cj{E_{n+1}(z)}\,\diff z,\qquad k\geq n\geq 0.
\end{align}
Using the complex version of Green's formula in the domain $\Omega$ 
(see the proof of Lemma 2.3 of \cite{Nikos1}), one finds the alternative expression 
\begin{align}\label{secondexpressionforvep}
\vep_{k,n}=(n+1)\int_{\Omega}H_k(z)\cj{H_{n}(z)}\,\diffA(z).
\end{align}
In particular, 
\begin{align}\label{vepnn}
 \vep_{n,n}=(n+1)\|H_{n}\|^2_{L^2(\Omega)}.
\end{align}
One can also verify without much difficulty that 
\begin{align}\label{vepnn-2}
 \vep_{n,n}=1-(n+1)\|G_n\|^2_{L^2(D)}.
\end{align}
It was also proved in \cite[Eq. (2.22 )]{Nikos1} that 
\begin{align}\label{betann}
\beta_{n,n}=(n+1)\|q_{n-1}\|^2_{L^2(D)}.
\end{align}

Let us now define, for integers $n,k\geq 0$, the quantities
\begin{align}\label{definitionalphas}
\alpha_{n,k}:= -\frac{1}{2\pi i}\cj{\int_{L}p_k(\zeta)\phi(\zeta)^{-n-1}\diff\zeta}.
\end{align}
By an application of Cauchy's theorem, it readily follows that 
\begin{align}\label{alphavaluesfrom-one-to-n}
\alpha_{n,k}=\begin{cases}
0, &\ 0\leq k<n,\\
\lambda_n/\gamma^{n+1}, &\ k=n.
\end{cases}
\end{align}

\begin{lemma}\label{estimatesforalphas}If $L$ is rectifiable, then we have 
\begin{align}
\frac{\gamma^{n+1}}{\lambda_n}\cj{\alpha_{n,n}}={}& (n+1)\frac{\gamma^{2(n+1)}}{\gamma^2_n}+\beta_{n,n}+\vep_{n,n}, \qquad n\geq 0,\label{est1}\\
\frac{\gamma^{k+1}}{\lambda_k}\cj{\alpha_{n,k}}={} &\beta_{k,n}+\vep_{k,n},\qquad k>n\geq 0. \label{est2}
\end{align}	
\end{lemma}

\begin{proof} We assume throughout the proof that $k\geq n\geq 0$. We have 
\begin{align}\label{expalpha1}\begin{split}
\cj{\alpha_{n,k}}={} &\frac{1}{2\pi i}\int_{L}p_k(\zeta)\cj{F_{n+1}(\zeta)}\,\diff\zeta
+\frac{1}{2\pi i}\int_{L}p_k(\zeta)\cj{E_{n+1}(\zeta)}\,\diff\zeta\\
={} & \int_{D}p_k(z)\cj{F'_{n+1}(z)}\,\diffA(z)+\frac{1}{2\pi i}\int_{L}p_k(\zeta)\cj{E_{n+1}(\zeta)}\,\diff\zeta\\
={}& (n+1) \int_{D}p_k(z)\cj{G_{n}(z)}\,\diffA(z)+\frac{1}{2\pi i}\int_{L}p_k(\zeta)\cj{E_{n+1}(\zeta)}\,\diff\zeta.
\end{split}
\end{align}
By the orthonormality property of the polynomials $p_n$, we  have
\begin{align*}
\int_{D}p_k(z)\cj{G_{n}(z)}\,\diffA(z)=\begin{cases}\gamma^{n+1}/\lambda_n,& k=n,\\
0, & k>n.	
\end{cases}
\end{align*}
It follows from the latter equality and \eqref{expalpha1} that 
\begin{align*}
\frac{\gamma^{k+1}}{\lambda_k}\cj{\alpha_{n,k}}={} &
\frac{\gamma^{k+1}}{\lambda_k}\cdot\frac{1}{2\pi i}\int_{L}p_k(\zeta)\cj{E_{n+1}(\zeta)}\,\diff\zeta
+\begin{cases}\gamma^{2(n+1)}/\lambda^2_n,& k=n,\\
0, & k>n.	
\end{cases}
\end{align*}
Thus, the proof of Lemma \ref{estimatesforalphas} will be complete once we show that 
\begin{align}\label{integralrel}
\frac{\gamma^{k+1}}{\lambda_k}\cdot\frac{1}{2\pi i}\int_{L}
p_k(\zeta)\cj{E_{n+1}(\zeta)}\,\diff\zeta={} &\beta_{k,n}+\vep_{k,n},\qquad k\geq n\geq 0. 
\end{align}	
Making the substitution
\begin{align*}
\frac{\gamma^{k+1}}{\lambda_k}p_k(z)=G_k(z)-q_{k-1}(z)
\end{align*}
in \eqref{integralrel} yields 
\begin{align}\label{onehalf}
\frac{\gamma^{k+1}}{\lambda_k}\cdot\frac{1}{2\pi i}\int_{L}p_k(\zeta)\cj{E_{n+1}(\zeta)}\,\diff\zeta
={} &\frac{1}{2\pi i}\int_{L}G_k(\zeta)\cj{E_{n+1}(\zeta)}\,\diff\zeta+\beta_{k,n}.
\end{align}
We can now use \eqref{defGnandHn} and \eqref{ident3} to compute
\begin{align*}
\frac{1}{2\pi i}\int_{L}G_k(\zeta)\cj{E_{n+1}(\zeta)}\,\diff\zeta
={}&\frac{1}{2\pi i}\int_{L}\phi(\zeta)^k\phi'(\zeta)\cj{E_{n+1}(\zeta)}\,\diff\zeta\\
&-\frac{1}{2\pi i}\int_{L}H_k(\zeta)\cj{E_{n+1}(\zeta)}\,\diff\zeta\\
={}&-\frac{1}{2\pi i}\int_{L}H_k(\zeta)\cj{E_{n+1}(\zeta)}\,\diff\zeta=\vep_{k,n},
\end{align*}
which together with \eqref{onehalf} yields \eqref{integralrel}, proving Lemma \ref{estimatesforalphas}.	
\end{proof}

Since $\alpha_{n,n}=\lambda_n/\gamma^{n+1}$, we  get from \eqref{est1} that 
\begin{align}\label{estimateforlambda-n}
	(n+1)\frac{\gamma^{2(n+1)}}{\lambda_n^2}=1-(\beta_{n,n}+\vep_{n,n}), \qquad n\geq 0.
\end{align} 
This important identity was  established in \cite[Lemma 2.4]{Nikos1}. It shows that for $L$ rectifiable,
\begin{align}\label{ineq:alpha_nn}
	\frac{\gamma^{n+1}}{\lambda_n}=	\frac{1}{\alpha_{n,n}}\leq \frac{1}{\sqrt{n+1}}, \qquad n\geq 0,
\end{align}
and that    
\[
\lim_{n\to\infty}(n+1)\frac{\gamma^{2(n+1)}}{\lambda_n^2}=1
\]
if and only if $\lim_{n\to\infty}\beta_{n,n}=\lim_{n\to\infty}\vep_{n,n}=0$. 

The quantity $\vep_{n,n}$ only depends on the geometry of the curve $L$ via the exterior map 
$\phi$, while $\vep_{n,n}$ still depends on the polynomial $p_n$. 
It turns out that if $L$ is also a quasiconformal curve, then $\beta_{n,n}$ can be 
uniformly bounded by $\vep_{n,n}$. This observation,  made precise in our next proposition,
was first realized in \cite[Thm. 2.1]{Nikos1}. 

A Jordan curve $L$ is quasiconformal provided that there is a constant $M$ such that $\diam L(z,\zeta)\leq M |z-\zeta|$ for all $z,\zeta\in L$, where  $L(z,\zeta)$ is the arc of $L$ between $z$ and $\zeta$ of smallest diameter. In particular, every piecewise analytic curve without cusps is quasiconformal.  For the definition of the reflection factor $\kappa$ of a quasiconformal curve and other details, see \cite{Nikos1}.

\begin{proposition} 
If $L$ is a rectifiable quasiconformal curve, with a reflection factor of  $\kappa$, then  
\begin{align}\label{ineq:beta-vep}
0\leq \beta_{n,n}\leq \frac{\kappa^2}{1-\kappa^2}\vep_{n,n},\qquad n\geq 0.
\end{align}
\end{proposition}

The proof of \eqref{ineq:beta-vep} (take $k=n$ in \eqref{temporary1} below)
is based on the following inequality taken from \cite[Lemma 2.5]{Nikos1}. 

\begin{proposition}Suppose $L$ is quasiconformal and rectifiable, and let $\kappa$ 
denote the reflection factor of $L$. For every $f$ analytic in $D$ and continuous on $\cj{D}$, 
and for every $g$ analytic in $\Omega$ and continuous on $\cj{\Omega}$, with $g'\in A^2(\Omega)$, we have  
\begin{align}\label{quasiconformalinequality}
\left|\frac{1}{2i\pi}\int_{L}f(z)\cj{g(z)}\,\diff z\right|
\leq \frac{ \kappa}{\sqrt{1-\kappa^2}}\|f\|_{L^2(D)}\|g'\|_{L^2(\Omega)}\,.
\end{align}
\end{proposition}

We will moreover find use for the following result.

\begin{proposition}
Suppose $L$ is quasiconformal and rectifiable, and let $\kappa$ 
denote the reflection factor of $L$. The following inequalities hold true:
\begin{align}\label{ineq:lambda-n}
0\leq 1-(n+1)\frac{\gamma^{2(n+1)}}{\lambda^2_n}\leq{} &\frac{1}{1-\kappa^2}\vep_{n,n},\qquad n\geq 0, 
\end{align}	
\begin{align}\label{ineq:alpha-nk}
\frac{\gamma^{k+1}}{\lambda_k}|\alpha_{n,k}|\leq {} &\frac{1}{1-\kappa^2}
\sqrt{\frac{\vep_{k,k}}{k+1}}\sqrt{(n+1)\vep_{n,n}},\qquad k>n\geq 0.
\end{align}	
\end{proposition}
 \begin{proof}The inequality \eqref{ineq:lambda-n} is a clear consequence of 
\eqref{estimateforlambda-n} and \eqref{ineq:beta-vep}.

To prove \eqref{ineq:alpha-nk}, we apply the triangle inequality to  \eqref{est2} to get  
\begin{align}\label{ineq:alpha-nk-2}
\frac{\gamma^{k+1}}{\lambda_k}|\alpha_{n,k}|\leq {} &|\beta_{k,n}|+|\vep_{k,n}|,\qquad k>n\geq 0,
\end{align}	
and proceed to estimate the right-hand side of \eqref{ineq:alpha-nk-2}. 
First, we apply the inequality \eqref{quasiconformalinequality} to the integral that defines 
$\beta_{k,n}$ in \eqref{defbeta-n} to deduce that (recall \eqref{vepnn} and \eqref{betann})
\begin{align*}
|\beta_{k,n}|={} &	\left|\frac{1}{2\pi i}\int_{L}q_{k-1}(\zeta)\cj{E_{n+1}(\zeta)}\,\diff\zeta\right|\\
\leq{} & \frac{\kappa}{\sqrt{1-\kappa^2}}\|q_{k-1}\|_{L^2(\Omega)}(n+1)	\|H_{n+1}\|_{L^2(D)}\\
\leq{}&  \frac{\kappa}{\sqrt{1-\kappa^2}}\sqrt{\frac{\beta_{k,k}}{k+1}}\sqrt{(n+1)\vep_{n,n}},
\end{align*}
which in view of \eqref{ineq:beta-vep} yields 
\begin{align}\label{temporary1}
|\beta_{k,n}|\leq {} &  \frac{\kappa^2}{1-\kappa^2}	\sqrt{\frac{\vep_{k,k}}{k+1}}\sqrt{(n+1)\vep_{n,n}}.
\end{align}
Next, we take absolute values in \eqref{secondexpressionforvep} and apply the Cauchy-Schwarz inequality to obtain   
\begin{align}\label{temporary2}
\begin{split}
|\vep_{k,n}|={} &(n+1)\left|\int_{\Omega}H_k(z)\cj{H_{n}(z)}\,\diffA(z)\right|\\
\leq {} &(n+1)\|H_k\|_{L^2(\Omega)}\|H_{n+1}\|_{L^2(\Omega)}\\
\leq {} &\sqrt{\frac{\vep_{k,k}}{k+1}}\sqrt{(n+1)\vep_{n,n}}.
\end{split}
\end{align}
The inequality  \eqref{ineq:alpha-nk} now follows by replacing $|\beta_{k,n}|$ and 
$|\vep_{k,n}|$ in \eqref{ineq:alpha-nk-2} by their upper bounds found in \eqref{temporary1} and \eqref{temporary2}.
\end{proof}

Having established \eqref{ineq:alpha-nk}, we then look for a good estimate for $\vep_{n,n}$. 
This is provided by the following result proven in \cite[Thm. 2.4]{Nikos1}.
\begin{proposition}\label{propositionvepnn}
If $L$ is piecewise analytic without cusps, 
then there exists a constant $c_1$, that depends only on $L$, such that  
\[
\vep_{n,n}\leq\frac{c_1}{n+1},\qquad n\geq 0.
\] 
\end{proposition}
As a consequence of Proposition \ref{propositionvepnn}, one obtains from \eqref{ineq:lambda-n} that
\begin{align}\label{estimateforgamma_n}
\alpha_{n,n}	=\frac{\lambda_n}{\gamma^{n+1}}=\sqrt{n+1}\big(1+\Ordo(1/n)\big).
\end{align} 
as $n\to\infty$.

\begin{corollary}	If $L$ is piecewise analytic without cusps, 
then there exists a constant $C$, that depends only on $L$, such that  
\begin{align}\label{estimateforalpha_n}
|\alpha_{n,k}|\leq \frac{C}{\sqrt{k+1}},\qquad k>n\geq 0.
\end{align} 
\end{corollary}
\begin{proof}
Using the estimate afforded by Proposition \ref{propositionvepnn} on the right-hand side of 
\eqref{ineq:alpha-nk}, we can find a constant $c_2$, depending only on $L$, such that
\begin{align*}
\frac{\gamma^{k+1}}{\lambda_k}|\alpha_{n,k}|\leq  {} &	\frac{c_2}{k+1},\qquad k>n\geq 0.
\end{align*}
This and \eqref{estimateforgamma_n} yield \eqref{estimateforalpha_n}.	
\end{proof}

We finish this section with a lemma about the asymptotic  behavior of the polynomials $G_n$ that were introduced in  \eqref{secondkind-Faber}. The result is a direct corollary of the asymptotic analysis carried out in \cite{MinaDiaz3} for the Faber polynomials associated to a domain $\Omega$ with piecewise analytic boundary.
\begin{lemma}\label{Lemma-on-Gn}Suppose that $L$ is a piecewise analytic Jordan curve without cusps, and let $C(L)$ denote the corners of $L$. For every $\epsilon>0$, there exists a constant $m_\epsilon$ such that 
	\begin{align*}
	\left|\frac{G_n(z)}{\phi'(z)\phi(z)^n}-1\right|\leq \frac{m_\epsilon}{n+1},\qquad z\in \caliE_\epsilon,\quad n\geq 0, 
	\end{align*}
where $\caliE_\epsilon:=\{z\in \Omega: {\rm d}_{\C}(z,C(L))>\epsilon\}$.
\end{lemma}

\begin{proof} In \cite{MinaDiaz3}, fine asymptotic formulas were derived for the Faber polynomials $F_n$ associated to a domain $\Omega$ whose boundary $L$ is  a piecewise analytic curve without inner cusps ($L$ not necessarily a Jordan curve). This is a  more general assumption than the one we make in Lemma \ref{Lemma-on-Gn}. Let $\lambda\pi\in (0,2\pi)$ be the smallest of the exterior angles formed at the corners of $L$. Since $\cj{\caliE}_\epsilon$ is a closed subset of $\cj{\Omega}\setminus C(L)$, Theorem 2.4 of \cite{MinaDiaz3}  implies that there exists an open set $U\supset \cj{\caliE}_\epsilon$ such that $\phi$ extends as a univalent function to $U$, and 
	\[
	F_{n+1}(z)=\phi(z)^{n+1} +\Ordo(n^{-\lambda}),\qquad z\in U,\quad n\geq 1,
	\]
uniformly on closed subsets of $U$. By differentiating this relation,  we find that  
\[
\frac{G_n(z)}{\phi'(z)\phi(z)^n}=\frac{F'_{n+1}(z)}{(n+1)\phi'(z)\phi(z)^n}=1+\Ordo(n^{-\lambda-1}|\phi(z)|^{-n}),\qquad n\geq 1,
\]
uniformly on closed subsets of $U$, and the lemma follows since $|\phi(z)|> 1$ for $z\in \caliE_\epsilon$. 
\end{proof}

\section{Proof of the asymptotic Results}
\label{s:proofs}

\subsection{Proof of Proposition \ref{PROP:BECKSTYL}}
\noindent The inequality \eqref{second-ineq-A_n} is essentially contained in the proof of 
Theorem 1.1 of \cite{BeckermannStylianopoulos},
but the authors restrict attention to the exterior asymptotics away from
the boundary. To obtain Proposition~\ref{PROP:BECKSTYL} we 
only need to go through their proof,
paying attention to precisely how their bounds degenerate 
as we approach $\partial \Omega$.

In terms of the polynomials $G_n$ introduced in \eqref{secondkind-Faber}, the function $A_n(z)$ can be expressed as
	\begin{align*}
		A_n(z)&=\left(\frac{G_n(z)}{\phi'(z)\phi(z)^n}-1\right)
		+\frac{p_n(z)-\sqrt{n+1}G_n(z)}{\sqrt{n+1}\phi'(z)\phi(z)^n}\\
		&=:A_n^1(z)+A_n^2(z).
	\end{align*}
	This is the same as \cite[Eq.\ (4.1)]{BeckermannStylianopoulos}, if we take into account that we are using a different normalization of the area measure, and that 
	the functions $f_n$ in  \cite{BeckermannStylianopoulos} are precisely defined as 
	\[
	f_n(z):=\frac{F_{n+1}'(z)}{\sqrt{\pi(n+1)}}=\sqrt{\frac{n+1}{\pi}}G_n(z).
	\]
	
Let us fix $\epsilon>0$ and let $\caliE_\epsilon=\{z\in \Omega: {\rm d}_{\C}(z,C(L))>\epsilon\}$.  By Lemma \ref{Lemma-on-Gn}, there exists a constant $m_\epsilon$ such that 
\begin{align}\label{Gn-first-ineq}
|A_n^1(z)|\leq \frac{m_\epsilon}{n+1},\qquad z\in \caliE_\epsilon,\quad n\geq 0.
\end{align}
In particular, this implies that 
\begin{align}\label{Gn-second-ineq}
\left|\frac{G_n(z)}{\phi'(z)\phi(z)^n}\right|\leq 1+m_\epsilon, \qquad z\in \caliE_\epsilon,\quad n\geq 0.
\end{align}
In view of the inequality (2.14) of \cite{BeckermannStylianopoulos} and the one preceding (4.3) therein, we have the bound
	\begin{align}
		\label{eq:bd-An2-prel}
		|A_n^2(z)|\leq{} &
		c_2	\frac{\sqrt{\varepsilon_{n,n}}}{\sqrt{n+1}}\sum_{j=0}^n\sqrt{\varepsilon_{j,j}}\sqrt{j+1} |\phi(z)|^{j-n}\left|\frac{G_j(z)}{\phi'(z)\phi(z)^j}\right|,\qquad z\in \Omega,
	\end{align}
	where $c_2$ is a certain constant, and the numbers $\varepsilon_{n,n}$ are those previously introduced in \eqref{defvarepsilon-kn} ($\varepsilon_{n,n}$ is called $\varepsilon_n$ in \cite{BeckermannStylianopoulos} and defined by the alternative expression \eqref{vepnn-2}).
		
	Since $L$ is piecewise analytic without cusps, we have in view of Proposition \ref{propositionvepnn} (see also \cite[Eq. (1.12)]{BeckermannStylianopoulos}) that $\varepsilon_{n,n}\leq c_1(n+1)^{-1}$ for some constant $c_1$ independent of $n$.	Incorporating this and \eqref{Gn-second-ineq} into \eqref{eq:bd-An2-prel}, we find 
	\begin{align*}
	|A_n^2(z)|\leq{} &\frac{
	c_1	c_2(1+m_\epsilon)}{n+1}\sum_{j=0}^n|\phi(z)|^{j-n},\qquad z\in \caliE_\epsilon,\quad n\geq 0,
\end{align*}
which together with \eqref{Gn-first-ineq} yields
	\begin{align*}
	|A_n^1(z)|+|A_n^2(z)|\leq{} & \frac{\left(m_\epsilon+c_1	c_2(1+m_\epsilon)\right)|\phi(z)|}{(n+1)(|\phi(z)|-1)}, \qquad z\in \caliE_\epsilon,\quad n\geq 0,
\end{align*}
completing the proof of Proposition \ref{PROP:BECKSTYL}.

\subsection{Proof of Theorem \ref{theorem:seriesexpansion}}\label{proof-thm-series-expansion}
\noindent The Bergman kernel of $A^2(D)$ is the function $K(z,\zeta)$ 
on $D\times D$ which is analytic in $z$ and conjugate-analytic in $\zeta$, with finite norm
\[
\int_D|K(z,\zeta)|^2\diffA(\zeta)<\infty,\qquad z\in D,
\]
and such that for any $f\in A^2(D)$, the reproducing property holds:
\[
f(z)=\int_Df(\zeta)K(z,\zeta)\,\diffA(\zeta).
\] 
This kernel can be expressed in terms of the conformal map $\vfi$ of $D$ onto $\D(0,1)$ by the formula 
\[
K(z,\zeta)=\frac{\varphi'(z)\cj{\varphi'(\zeta)}}{(1-\varphi(z)\cj{\varphi(\zeta)})^2}.
\]

It is well-known that the polynomials are dense in $A^2(D)$, and therefore 
\begin{align}\label{kerneloverconvergence}
\frac{\varphi'(z)\cj{\varphi'(\zeta)}}{(1-\varphi(z)\cj{\varphi(\zeta)})^2}
=\sum_{k=0}^\infty\cj{p_k(\zeta)}p_k(z),\qquad z,\zeta\in D.
\end{align}
Here it is important to know that this equality actually holds in a bigger domain.
Indeed, for $\zeta\in D$ fixed, we can easily see that $K(z,\zeta)$ is analytic in the variable $z$ 
in some open set that contains $\cj{D}$. Let $\tau(\zeta)<1$ be the smallest number such that 
$K(\cdot,\zeta)$ is analytic in the domain $\caliD_{1/\tau(\zeta)}$ (the sets $\caliD_r$ were 
previously defined in Subsection \ref{p_n-on-D}).
Then \eqref{kerneloverconvergence} holds for all $ \zeta\in D$ and $z\in \caliD_{1/\tau(\zeta)}$. 
Moreover, for $\zeta\in D$ fixed, the convergence is uniform for $z$ on compact subsets 
of $\caliD_{1/\tau(\zeta)}$. We also have 
\begin{align}\label{limisupidentity}
\limsup_{n\to\infty}|p_n(\zeta)|^{1/n}=\tau(\zeta)<1,\qquad \zeta\in D.
\end{align}
These stronger statements follow from the overconvergence results of Walsh, see
\cite[pp. 130-131]{Walsh} (see also \cite[p. 336]{PS}).

By the Hermitian symmetry of the kernel, for $z\in D$ fixed, we have uniform convergence in 
\eqref{kerneloverconvergence}  for $\zeta\in L$. We then multiply \eqref{kerneloverconvergence} by 
\[
\frac{\phi(\zeta)^{n+1}\varphi'(\zeta)\,\diff\zeta}{\cj{\varphi'(\zeta)}\,\varphi(\zeta)^2}
\]
and integrate with respect to $\zeta$ over $L=\partial D$ to get (bearing in mind that 
when $\zeta\in L$, we have $\cj{\varphi(\zeta)}=1/\varphi(\zeta)$)
\begin{align}\label{expansion}
\frac{\varphi'(z)}{2\pi i}\int_{L}\frac{\varphi'(\zeta)\phi(\zeta)^{n+1}}{(\varphi(\zeta)-\varphi(z))^2}\,\diff\zeta
={} & \sum_{k=0}^\infty \frac{p_k(z)}{2\pi i}
\int_{L}\frac{\cj{p_k(\zeta)}\phi(\zeta)^{n+1}\varphi'(\zeta)\,\diff\zeta}{\cj{\varphi'(\zeta)}\,\varphi(\zeta)^2},\qquad z\in D.
\end{align}
Now we note that 
\[
\cj{\diff\zeta}=-\frac{\varphi'(\zeta)\,\diff\zeta}{\cj{\varphi'(\zeta)}\,\varphi(\zeta)^2},
\]
meaning that for every function $f$ continuous on $L$,
\[
\cj{\int_Lf(\zeta)\,\diff\zeta}=\int_{L}\cj{f(\zeta)}
\underbrace{\left(-\frac{\varphi'(\zeta)\,\diff\zeta}{\cj{\varphi'(\zeta)}\,
\varphi(\zeta)^2}\right)}_{{\displaystyle ``\, \cj{\diff\zeta}\,"}}.
\]
We then integrate by parts on the left-hand side of \eqref{expansion} to get 
\begin{align*}
\frac{\varphi'(z)(n+1)}{2\pi i}\int_{L}\frac{\phi'(\zeta)\phi(\zeta)^n}{\varphi(\zeta)-\varphi(z)}\,\diff\zeta
={} &	\sum_{k=0}^\infty\alpha_{n,k}p_k(z),\qquad z\in D,
\end{align*}
where
\[
\alpha_{n,k}= -\frac{1}{2\pi i}\int_{L}\cj{p_k(\zeta)}\phi(\zeta)^{n+1}\cj{\diff\zeta}
= -\frac{1}{2\pi i}\cj{\int_{L}p_k(\zeta)\phi(\zeta)^{-n-1}\diff\zeta}.
\]
The coefficients $\alpha_{n,k}$ are the same quantities previously introduced in 
\eqref{definitionalphas}. From \eqref{ineq:alpha_nn} and  \eqref{estimateforalpha_n},
we know that  
\begin{align}\label{estimatestogether}
\frac{1}{\alpha_{n,n}}\leq \frac{1}{\sqrt{n+1}},\qquad |\alpha_{n,k}|\leq \frac{C}{\sqrt{k+1}},\qquad k>n\geq 0,
\end{align}	
where $C$ is a constant independent of $k$ and $n$.
We have also seen in \eqref{alphavaluesfrom-one-to-n} that $\alpha_{n,k}=0$ for $0\leq k<n$, so that we can write \begin{align}\label{firstexpansion}
\alpha_{n,n}p_n(z)={}& Q_n(z)-\sum_{k=n+1}^\infty\alpha_{n,k}p_k(z),
\end{align}
with $Q_n(z)$ as defined in \eqref{definitionQ_n}.

The identity \eqref{firstexpansion} is similar to the identity (6.5) obtained in \cite{MinaDiaz1} 
for polynomials orthogonal with weights on the unit disk. The remaining part of the proof is 
essentially a repetition of the steps carried out in the proofs of Lemmas 6.1 and 6.2 of \cite{MinaDiaz1}, 
but we will quickly summarize them for the sake of keeping the presentation
self-contained.

The idea is to iterate \eqref{firstexpansion} to eliminate the polynomials $p_{k}$ for all $k\ge n+1$,
and end up with the right-hand side of \eqref{firstexpansion} written as a series in the $Q_n$'s. 
The  identity \eqref{firstexpansion} for $n+1$ in place of $n$ reads 
\begin{align}\label{firstexpansion1}
p_{n+1}(z)={}& \frac{1}{\alpha_{n+1,n+1}}Q_{n+1}(z)-\frac{1}{\alpha_{n+1,n+1}}\sum_{k=n+2}^\infty\alpha_{n+1,k}p_k(z).
\end{align}
If we now substitute the right-hand side of \eqref{firstexpansion1} for $p_{n+1}$ in the series of  \eqref{firstexpansion}, 
we get a new series where the only polynomials $p_k$ that remain are those of degree $k\geq n+2$. 
Repeating this process $m$ times, we obtain (cf. Lemma 6.1 of \cite{MinaDiaz1})
 \begin{align}\label{polysecondexpansion}
\alpha_{n,n}p_n(z)={} &\sum_{j=0}^mh(n,j)Q_{n+j}(z)+\sum_{k=n+m+1}^\infty g(n,m,k)p_k(z),
\end{align}
where the coefficients $h(n,j)$ and $g(n,m,k)$ are recursively defined by the relations 
\begin{align}\label{numbers-h-and-g-1}
h(n,0)=1,\quad g(n,0,k)=-\alpha_{n,k},\qquad k> n\geq 0,
\end{align}
\begin{align}\label{numbers-h-and-g-2}
h(n,m+1)=\frac{g(n,m,n+m+1)}{\alpha_{n+m+1,n+m+1}}\,,
\end{align}
and
\begin{align}\label{numbers-h-and-g-3}
g(n,m+1,k)= g(n,m,k)-h(n,m+1)\alpha_{n+m+1,k},\qquad k\geq n+m+2.
\end{align}
Combining \eqref{numbers-h-and-g-1}, \eqref{numbers-h-and-g-2}, and \eqref{numbers-h-and-g-3}  with \eqref{estimatestogether} we find 
\begin{align}\label{ineq1}
h(n,0)=1,\quad |g(n,0,k)|\leq  \frac{C}{\sqrt{k+1}},\qquad k\geq n+1,
\end{align}
\begin{align}
|h(n,m+1)|\leq {} &\frac{|g(n,m,n+m+1)|}{\sqrt{n+m+2}},\label{ineq2}
\end{align}
\begin{align}
|g(n,m+1,k)|\leq {} & |g(n,m,k)|+\frac{C}{\sqrt{k+1}}|h(n,m+1)|,\label{ineq3}
\end{align}
with the inequality \eqref{ineq3} being valid for $k\geq n+m+2$. 

By induction in the second variable $m$, we get from \eqref{ineq1}, \eqref{ineq2} 
and \eqref{ineq3} that (c.f. Lemma 6.2 of \cite{MinaDiaz1} and its proof)	
\begin{align}\label{bounds-for-hn-2}
|h(n,j)|\leq \frac{C}{n+j+1}\prod_{\ell=1}^{j-1}\left(1+\frac{C}{n+\ell}\right) ,\qquad j\geq 1,
\end{align}		
\begin{align}\label{inequalityforg}
|g(n,m,k)|\leq \frac{C}{\sqrt{k+1}}\prod_{\ell=1}^{m}\left(1+\frac{C}{n+\ell}\right),\qquad m\geq 0,\ k\geq n+m+1.
\end{align}

We now write
\begin{align}\label{estimate-for-product}
\prod_{\ell=1}^{m}\left(1+\frac{C}{n+\ell}\right)={} 
&\exp\left(\sum_{\ell=1}^{m}\log\left(1+\frac{C}{n+\ell}\right)\right)\nonumber\\
\leq {} &\exp\left(\sum_{\ell=1}^{m}\frac{C}{n+\ell}\right)=\left(1+\frac{m}{n+1
}\right)^C\exp(C\varrho_{n,m})\nonumber\\
<{} &e^C\left(1+\frac{m}{n+1}\right)^C,
\end{align}	
where 
\begin{align*}
\varrho_{n,m}=\left(\sum_{k=1}^{n+m}\frac{1}{k}-\log(n+m+1)\right)
-\left(\sum_{k=1}^{n}\frac{1}{k}-\log (n+1)\right)< 1. 
\end{align*}	
Replacing the product in \eqref{bounds-for-hn-2} by the estimate \eqref{estimate-for-product} 
(with $m=j-1$) readily yields \eqref{bounds-for-hn} with the choice of $B=Ce^C$. 
Similarly, applying \eqref{estimate-for-product} to \eqref{inequalityforg} we  obtain 
\begin{align}\label{inequalityforgV2}
|g(n,m,k)|\leq \frac{B}{\sqrt{k}}\left(1+\frac{m}{n}\right)^B ,\qquad m\geq 0,\ k\geq n+m+1.
\end{align}
	
To finish the proof of Lemma~\ref{polythirdexpansion} we just need to show that the second 
series in \eqref{polysecondexpansion} converges to zero as $m\to\infty$ for $z\in D$. 
Let us then think of $z\in D$ as being fixed. Since  \[
\tau(z)=\limsup_{k\to\infty}|p_k(z)|^{1/k}<1,
\] 
we can find $\epsilon>0$ and a constant $M_\epsilon$ such that $\tau(z)+\epsilon<1$ and 
$|p_k(z)|\leq M_{\epsilon} (\tau(z)+\epsilon)^k$ for all $k\geq 0$. 
This together with \eqref{inequalityforgV2} imply that 
\begin{align*}
\sum_{k=n+m+1}^\infty |g(n,m,k) & p_k(z)|\\
&\leq M_\epsilon B\left(1+\frac{m}{n}\right)^B \sum_{k=n+m+1}^\infty \frac{(\tau(z)+\epsilon)^k}{\sqrt{k}}\\
&\leq \frac{M_\epsilon B (\tau(z)+\epsilon)^{n+1}}{\sqrt{n+m+1}
(1-(\tau(z)+\epsilon))}\left(1+\frac{m}{n}\right)^B(\tau(z)+\epsilon)^m.
\end{align*}
In view of the convergence
\[
\lim_{m\to\infty}\left(1+\frac{m}{n}\right)^B(\tau(z)+\epsilon)^m=0,
\]
we conclude that for $z\in D$, we have 
\[
\lim_{m\to\infty}\sum_{k=n+m+1}^\infty g(n,m,k)p_k(z)=0,
\]
as desired.
	
\subsection{Proof of Theorem \ref{theorem:strongasymptotic}}
\noindent We begin by making the change of variable $\zeta=\psi(w)$ in the integral 
on the right-hand side of \eqref{definitionQ_n}. This yields 
\begin{align}\label{secondrepresentationforQ_n}
Q_n(z)={} &	\frac{(n+1)\varphi'(z)}{2\pi i}\int_{|w|=1}\frac{w^n}{f_z(w)}\,\diff w,\qquad z\in D.
\end{align}
We will use \eqref{polythirdexpansion} to derive the behavior of $p_n$ as $n\to \infty$ 
from that of $Q_n$. Broadly speaking, the dominant behavior of the integral in 
\eqref{secondrepresentationforQ_n} is found by computing the residues of its 
integrand at the zeros of $f_z(w)$ of largest modulus.

We recall that 	$z\in D_p$ means that among the zeros of the meromorphic function 
\[
f_{z}(w):=h(w)-\varphi(z)
\]
in $\mu<|w|<1$, those of largest modulus, say  $w_{z,1},\ldots, w_{z,s}$, 
have a total multiplicity of $p$. That is, if $m_{z,k}$ denotes the 
multiplicity of $h$ at $w_{z,k}$, then $\sum_{k=1}^s m_{z,k}=p$.

The following lemma is proven by a standard application of Rouché's theorem. 
It is practically a repetition of Lemma 11 of \cite{DragnevMinaDiaz1}. 
\begin{lemma}\label{descriptionD_p}
Let $z_0\in D_p$ be fixed, and let $\mu_0$ be chosen in the range $\mu<\mu_0<r(z_0)$ such that
the only zeros of $f_{z_0}(w)$ in the region $\mu_0\leq  |w|<1$ are precisely those of 
largest modulus $w_{z_0,1},\ldots, w_{z_0,s}$, $1\leq s\leq p$. 
Let $\delta>0$ be small enough that the closed disks $\cj{\D(w_{z_0,k},\delta)}$, 
$1\leq k\leq s$, are pairwise disjoint, contained in the annulus $\mu_0<|w|<1$, and such that 
$h'(w)\not=0$ whenever $0<|w-w_{z_0,k}|\leq \delta$ for some  $1\leq k\leq s.$
Then there exists $\epsilon >0$ such that $\cj{\D(z_0,\epsilon)}\subset D$ and 
for every $z\in \D(z_0,\epsilon)$, the zeros of  $f_{z}(w)$ that lie in $\mu_0<|w|<1$ 
are all simple and contained in $\cup_{k=1}^s\D(w_{z_0,k},\delta)$, with each disk 
$\D(w_{z_0,k},\delta)$ containing a number $m_{z_0,k}$ of them.
\end{lemma}

It is easy to conclude from Lemma \ref{descriptionD_p} (c.f. Corollary 12 in \cite{MinaDiaz1}) 
that $D_1$ and $D\setminus D_0$ are open, that $\phi_1(z):=w_{z,1}$ is univalent on $D_1$, 
and that $r(z)$ is continuous on $D$. 

In order to prove Theorem \ref{theorem:strongasymptotic}, it suffices to show that every $z_0\in D_1$ 
has a neighborhood $\D(z_0,\epsilon)$ such that $A_n(z)=\Ordo(n^{-1})$ uniformly for $z\in \D(z_0,\epsilon)$ 
as $n\to\infty$. We then fix $z_0\in D_1$ and pick $\delta$ and $\epsilon$ as specified in 
Lemma \ref{descriptionD_p} with the choice $p=1$. In particular, we have 
\begin{align}\label{twobounds}
\mu_0<|\phi_1(z_0)|-\delta< |\phi_1(z)|< |\phi_1(z_0)|+\delta<1, \qquad z\in  \D(z_0,\epsilon).
\end{align}

For every $z \in \D(z_0,\epsilon)$, the function $f_z(w)$ has a single zero in $\mu_0\leq |w|\leq 1$, 
say $w_{z,1}=\phi_1(z)$, which is simple and contained in $\D(w_{z_0,1},\delta)$. 
It follows that the integrand in \eqref{secondrepresentationforQ_n} has a simple 
pole at $w_{z,1}$ and is analytic at any other point of the annulus $\mu_0\leq |w|\leq 1$.  
We can then apply the residue theorem to deduce that for all $z \in \D(z_0,\epsilon)$,
\begin{align}\label{secondrepresentationforQ_n1}
Q_n(z)=	(n+1)\varphi'(z)\frac{\phi_1(z)^n}{f_z'(\phi_1(z))}+
\frac{(n+1)\varphi'(z)}{2\pi i}\int_{|w|=\mu_0}\frac{w^n}{f_z(w)}\,\diff w.
\end{align}

The fact that $\phi_1(z)$ is a zero of $f_z(w)$ means that  $h(\phi_1(z))=\varphi(z)$. 
By differentiating this identity we obtain 
\begin{align}\label{derivativeidentity}
f_z'(\phi_1(z))= h'(\phi_1(z))=\frac{\varphi'(z)}{\phi_1'(z)},\qquad z\in \D(z_0,\epsilon).
\end{align} 

Since $|h(w)-\varphi(z)|:\T(0,\mu_0)\times \cj{\D(z_0,\epsilon)}\to (0,\infty]$ is 
continuous in $(w,z)$ and never vanishes, and since $|\vfi'|$ is bounded on $\cj{\D(z_0,\epsilon)}$, 
we deduce from \eqref{secondrepresentationforQ_n1} and \eqref{derivativeidentity} that   
\begin{align}\label{thirdrepresentationforQ_n}
Q_n(z)={} &(n+1)\phi_1'(z)\phi_1(z)^n+ \Ordo((n+1)\mu_0^n),\qquad z\in \D(z_0,\epsilon),
\end{align}
with the constant involved in the $\Ordo$-term being independent of $n$ and $z$. 
 We now combine \eqref{thirdrepresentationforQ_n} with \eqref{polythirdexpansion} 
and \eqref{estimateforgamma_n} to find
\begin{align}\label{fourthrepresentationforQ_n}
\begin{split}
A_n(z)={} &\frac{1+\Ordo(1/n)}{n+1}\sum_{j=1}^\infty h(n,j)(n+j+1)\phi_1(z)^{j}\\
&+\Ordo\left(\tau_0^n\sum_{j=0}^\infty \frac{h(n,j)(n+j+1)}{n+1}\mu_0^{j}\right)+\Ordo(1/n)
\end{split}
\end{align}
uniformly for $z\in \D(z_0,\epsilon)$ as $n\to\infty$, where $\tau_0=\mu_0/(|\phi_1(z_0)|-\delta)<1$.

By \eqref{bounds-for-hn}, we have that if $|\zeta|<1$, then 
\begin{align*}
\sum_{j=1}^\infty|h(n,j)|(n+j+1)|\zeta|^{j}	\leq {} & 
B\sum_{j=1}^\infty\left(1+\frac{j-1}{n}\right)^B|\zeta|^{j}\\
\leq {} & B2^{B}\left(\sum_{j=1}^\infty|\zeta|^{j}+
\sum_{j=1}^\infty\left(\frac{j-1}{n}\right)^B|\zeta|^{j}\right)\\
\leq {} & \frac{B2^{B}}{1-|\zeta|}+\frac{B2^{B}}{n^{B}}\sum_{k=0}^\infty k^B|\zeta|^{k}.
\end{align*}
By increasing $B$ if necessary, we may assume that $B$ is a positive integer. Making use of the identity 
\begin{align*}
\sum_{k=0}^\infty (k+1)\cdots (k+B)\zeta^j=\left(\frac{\zeta^B}{1-\zeta}\right)^{(B)},
\end{align*}	
one easily verifies that 
\begin{align*}
\sum_{k=0}^\infty (k+1)\cdots (k+B)|\zeta|^k\leq \frac{E}{(1-|\zeta|)^{B+1}}, \qquad |\zeta|<1,
\end{align*}
for some constant $E$ that only depends on $B$. Therefore, 
\begin{align}\label{estimado2}
\begin{split}
\sum_{j=1}^\infty	\frac{|h(n,j)|(n+j+1)}{n+1}|\zeta|^{j}
\leq  \frac{E_1}{n(1-|\zeta|)}+\frac{E_1}{n^{B_1}(1-|\zeta|)^{B_1}},
\end{split}
\end{align}
with $E_1=B2^BE$ and $B_1=B+1$.
Since $\mu_0<1$ and $|\phi_1(z)|<1$ for $\zeta\in D_1$, we can use \eqref{twobounds} 
and \eqref{estimado2} to estimate the  right-hand side of \eqref{fourthrepresentationforQ_n} 
and conclude that $A_n(z)=\Ordo(n^{-1})$ uniformly for $z\in  \D(z_0,\epsilon)$ as $n\to\infty$.

It remains to prove \eqref{limsup-equality}. Fix $z\in D$. By the very definition of $r(z)$ in \eqref{eq5}, 
the function $1/f_z(w)$ is analytic in the annulus $r(z)<|w|<1$, with a singularity on 
the circle $|w|=r(z)$ in case $r(z)>0$. We then look at its Laurent expansion 
\begin{align*}
\frac{1}{f_z(w)}=	\sum_{k=-\infty}^\infty a_k(z)w^k, \qquad r(z)<|w|<1,
\end{align*}
whose coefficients 
\begin{align*}
a_{-n}(z)=\frac{1}{2\pi i}\int_{|w|=1}\frac{w^{n-1}}{f_z(w)}\,\diff w,\qquad n\geq 0, 
\end{align*}
satisfy  
\[
\limsup_{n\to\infty}|a_{-n}(z)|^{1/n}=r(z).
\]
By \eqref{secondrepresentationforQ_n}, we have 
\begin{align*}
\limsup_{n\to\infty}|Q_n(z)|^{1/n}=r(z).
\end{align*}
Let  $\tau(z)=\limsup_{n\to\infty}|p_n(z)|^{1/n}<1$ (recall \eqref{limisupidentity}).
We need to show that $\tau(z)=r(z)$, or equivalently, that for every $\epsilon>0$,
\begin{align}\label{ineq:limisup}
r(z)-\epsilon\leq \tau(z)\leq 	r(z)+\epsilon.
\end{align}
Let $N_\epsilon$ be an index such that 
\begin{align*}
|Q_n(z)|^{1/n}\leq r(z)+\epsilon<1,\qquad |p_n(z)|^{1/n}< \tau(z)+\epsilon<1,
\end{align*}
for all $n>N_\epsilon$. 
Combining \eqref{polythirdexpansion}, \eqref{ineq:alpha_nn}, 
and \eqref{estimado2}, we find that for all $n >N_\epsilon$,
\begin{align*}
|p_n(z)|\leq{} &\frac{1}{\sqrt{n+1}}\sum_{j=0}^\infty |h(n,j)| |Q_{n+j}(z)|\\
\leq {} &\frac{(r(z)+\epsilon)^n}{\sqrt{n+1}}\left(1+\Ordo(1/n)\right),
\end{align*}
which  yields the second inequality in \eqref{ineq:limisup}. 
Similarly, we use \eqref{firstexpansion}, \eqref{estimatestogether}, 
and \eqref{estimateforgamma_n} to find that for all $n >N_\epsilon$,
\begin{align*}
|Q_n(z)|\leq {} &	\alpha_{n,n}|p_n(z)|+\sum_{k=n+1}^\infty|\alpha_{n,k}||p_k(z)|\\
\leq {} & \sqrt{n+1}(\tau(z)+\epsilon)^n\left(1+\Ordo(1/n)\right),
\end{align*}
which  yields the first inequality in \eqref{ineq:limisup}.

\subsection{Interior asymptotic estimates near \texorpdfstring{$L$}{}}
\label{section:asymptotics-near-boundary}
\noindent Let us denote the corners of $L$ by $z_1,\ldots,z_q$. Since $\phi$ takes $L$ onto the unit circle, 
we can write $\phi(z_k)=e^{i\theta_k}$, with $\theta_k \in [0,2\pi)$, $1\leq k\leq q$. 
For $\epsilon>0$ small enough to ensure that, for all $1\leq k\leq q$,
\begin{align}\label{angularcondition}
\{e^{i\theta}:\theta_k-\epsilon\leq\theta\leq\theta_k+\epsilon\}
\cap\{e^{i\theta_1},\ldots,e^{i\theta_q}\}=\{e^{i\theta_k}\},
\end{align}
and for every $0<r<1$, we define
\begin{align}\label{def:A}
V(r,\epsilon):={} & \bigcap_{k=1}^q\left\{te^{i\theta}:r< t< \frac{1}{r},\ 
\theta_k+\epsilon<\theta<\theta_k-\epsilon+2\pi\right\},
\end{align}
\begin{align}\label{def:V}
V_-(r,\epsilon):=V(r,\epsilon)\cap \D(0,1),\qquad V_+(r,\epsilon):=V(r,\epsilon)\cap \Delta(0,1).
\end{align}

\begin{figure}
	\begin{center}
		\scalebox{1}{\begin{tikzpicture}
				\draw (-4, 0) node[inner sep=0]{\includegraphics[scale=.8]{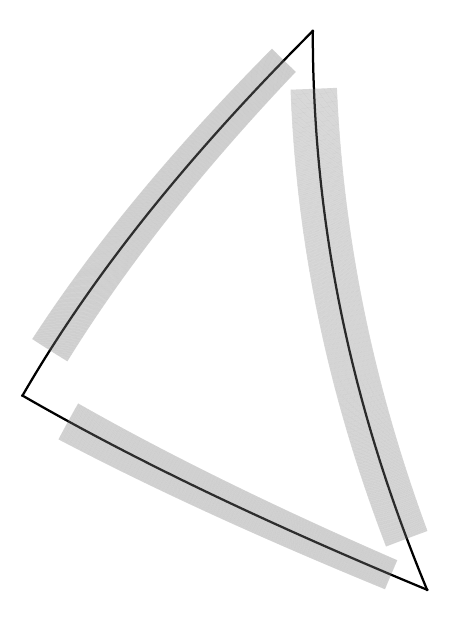}};
				\draw (3.5, 0) node[inner sep=0]{\includegraphics[scale=.9]{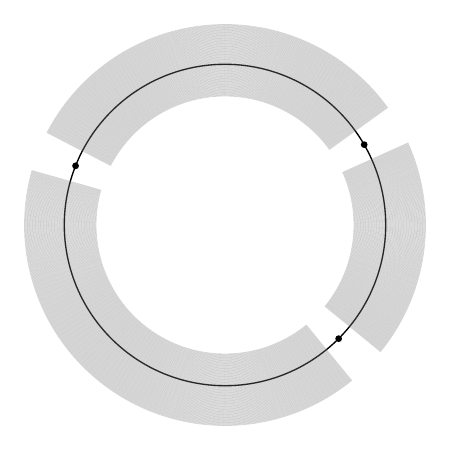}}; 
				\draw (-2.65, 4) node {$z_1$};
				\draw (-7.05, -1.2) node {$z_2$};
				\draw (-.97, -3.9) node {$z_3$};
				\draw (.93,1.05) node {$e^{i\theta_2}$};
				\draw (6.05,1.45) node {$e^{i\theta_1}$};
				\draw (5.51,-1.93) node {$e^{i\theta_3}$};
					\draw (-6.5,2.7) node {$U(r,\epsilon)$};
						\draw (5.5,3.3) node {$V(r,\epsilon)$};
				\draw[-Stealth]        (4.8,3.3)   -- (3.8,2.7);
				\draw[-Stealth]        (-5.75,2.7)   -- (-4.68,2);
		\end{tikzpicture}}
	\end{center}
\caption{Illustration of the open sets $V(r,\epsilon)$ and $U(r,\epsilon)$.}
\label{neighUandV}
\end{figure}

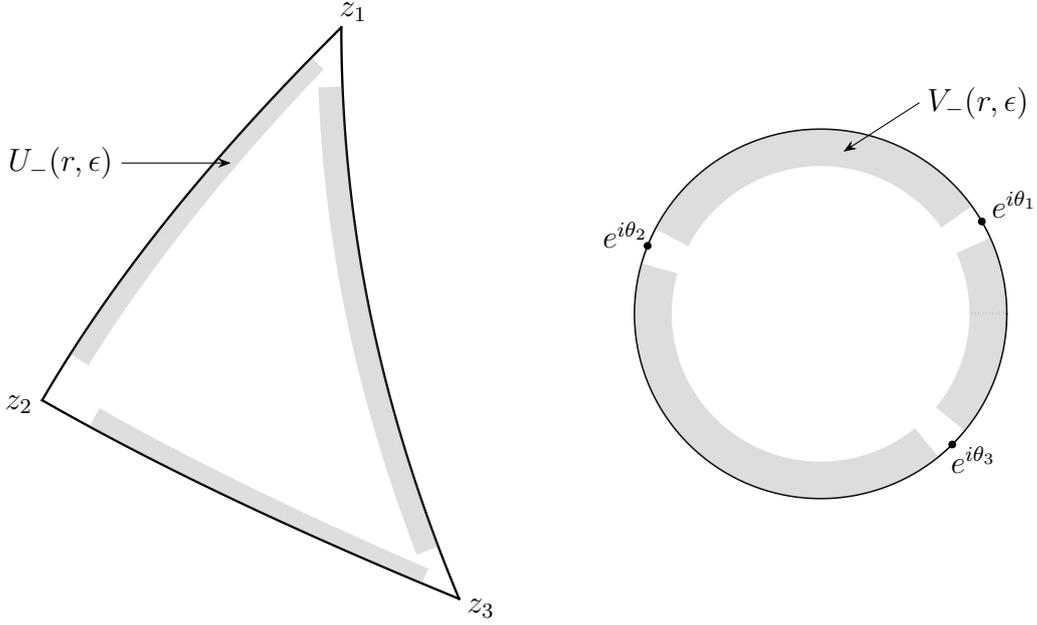
\begin{figure}
	\begin{center}
		\scalebox{1}{\begin{tikzpicture}
				\draw (-4, 0) node[inner sep=0]{\includegraphics[scale=.8]{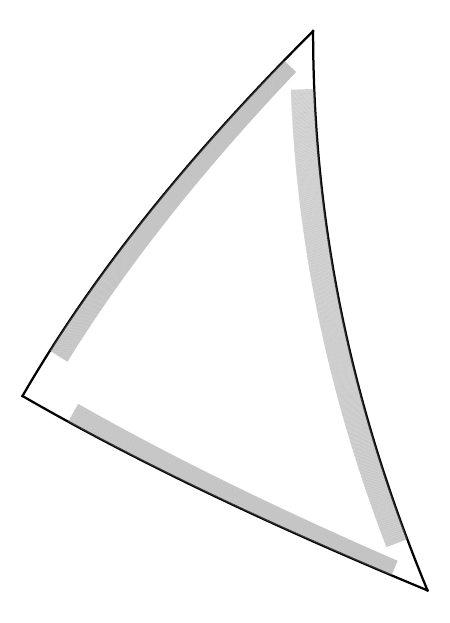}};
				\draw (3.5, 0) node[inner sep=0]{\includegraphics[scale=.9]{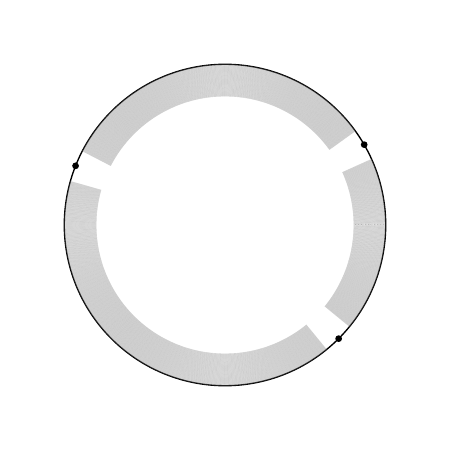}}; 
				\draw (-2.65, 4) node {$z_1$};
				\draw (-7.05, -1.2) node {$z_2$};
				\draw (-.97, -3.9) node {$z_3$};
				\draw (.93,1.05) node {$e^{i\theta_2}$};
				\draw (6.05,1.45) node {$e^{i\theta_1}$};
				\draw (5.51,-1.93) node {$e^{i\theta_3}$};
				\draw (-6.5,2) node {$U_-(r,\epsilon)$};
				\draw (5.57,2.8) node {$V_-(r,\epsilon)$};
				\draw[-Stealth]        (4.8,2.8)   -- (3.8,2.15);
				\draw[-Stealth]        (-5.7,2)   -- (-4.26,2);
		\end{tikzpicture}}
	\end{center}
\caption{Illustration of the open sets $V_-(r,\epsilon)$ and $U_{-}(r,\epsilon)$.}	\label{neighUandVminus}
\end{figure}

Note that $V(r,\epsilon)$ is its own reflection about the unit circle. Because $L$ is piecewise analytic, the reflection principle (see e.g. \cite[Ch. 8]{Davis}) allows us to find $0<r_\epsilon<1$ 
such that the map $\psi$ has a conformal extension to $\Delta(0,1)\cup V(r_\epsilon,\epsilon)$, 
which we also denote by $\psi$. Then for all $ r_\epsilon\leq r<1$, we set
\begin{align}\label{def:U}
U(r,\epsilon):=\psi(V(r,\epsilon)),\qquad U_\pm(r,\epsilon):=\psi(V_\pm(r,\epsilon)),
\end{align}
see Figures~\ref{neighUandV} and \ref{neighUandVminus}. 
The inverse of $\psi$ is the conformal extension of $\phi$ to 
$\Omega\cup U(r_\epsilon,\epsilon)$, which we also denote by $\phi$.  

\begin{proposition}\label{prop-asympt-near-boundary} 
For every $\epsilon>0$ satisfying \eqref{angularcondition}, there exists $r_\epsilon<r_0<1$ such that:
\begin{enumerate}[leftmargin=.9cm]
\item $U_-(r_0,\epsilon)\subset D_1$ and $\phi_1(z)=\phi(z)$ for all $z\in U_-(r_0,\epsilon)$. 
\item There exists positive constants $M_1>0$ and $B_1>1$, 
with $B_1$ independent of $\epsilon$ and $r_0$, such that
\begin{align}\label{growth-at-boundary}
\left|A_n(z)\right|\leq {} &\frac{M_1}{n(1-|\phi_1(z)|)}+\frac{M_1}{n^{B_1}(1-|\phi_1(z)|)^{B_1}}
\end{align}
for all $z\in U_-(r_0,\epsilon)$ and $n\geq 1$.  
\end{enumerate}
\end{proposition}

\begin{proof} 
Recall that  the function $h(w)$ is the meromorphic extension to the annulus $\mu<|w|<1/\mu$ of 
the homeomorphism of the unit disk $\vfi\circ \psi:\T(0,1)\to \T(0,1)$. Since $\varphi$ is analytic 
on $\cj{D}$, we can choose $0<\rho <1$ sufficiently close to $1$ so that $\varphi$ extends analytically 
to the interior domain  of the level curve $L_{1/\rho}=\{z\in \C:|\phi(z)|=1/\rho\}$ without ever 
attaining the value $0$ in said domain. Then  
\[
h(w)=\varphi(\psi(w)),\qquad  1\leq |w|<1/\rho,
\]
and by the reflection principle, $h(w)$ is also analytic in  $\rho<|z|<1/\rho$. 
Let us fix $\epsilon>0$ satisfying \eqref{angularcondition}, and let $\epsilon':=\epsilon/2$. 
We choose $r_{\epsilon'}<r'<1$ such that $r'>\rho$. 
By the reflection principle, the  map $\vfi$ is conformal on $D\cup U(r',\epsilon')$, 
so that the composite $\vfi(\psi(w))$ (which then coincides with $h(w)$) is univalent on $V(r',\epsilon')$. 

We now pick  $r_0$ such that $r'<2r_0-1<1$, and observe that for any such choice, 
we have $\cj{U(r_0,\epsilon)}\subset U(r',\epsilon')$, 
and since $\phi$ and $\varphi$ are conformal on $U(r',\epsilon')$, then 
\begin{align}\label{constant-c}
\beta_0:=	\max\left\{|\vfi'(z)|/|\phi'(z)|: z\in \cj{U_-(r_0,\epsilon)}\right\}<\infty.
\end{align} 
We claim that if $r_0$ is sufficiently close to $1$, then  for all $z\in \cj{U_-(r_0,\epsilon)}$, 
the equation $h(w)-\varphi(z)=0$ has at most one solution in  $2r_0-1\leq |w|\leq 1$. 
For otherwise, we are able to find a sequence $r_n\nearrow 1$ and a corresponding sequence 
of points $z_n\in \cj{U_-(r_n,\epsilon)}$, such that $z_n=h(w_{n,1})=h(w_{n,2})$, 
where $w_{n,1}$ and $w_{n,2}$ are two distinct points of the annulus $2r_n-1\leq |w|\leq 1$. 

Since $h(w)=\vfi(\psi(w))$ on $V(r',\epsilon')$, and since  $\cj{U_-(r_n,\epsilon)}\subset U(r',\epsilon')$, 
we may assume that one of these solutions, say $w_{n,1}$, is $\phi(z_n)$, 
which lies in $ \cj{V_-(r_n,\epsilon)}$. Therefore, there exists $n_0$ such that 
\[
w_{n,1}\in\cj{ V_-((1+r')/2,\epsilon)}\subset V(r',\epsilon'),\qquad n>n_0.
\]
Note that this $\in$ relation is also satisfied by any limit point of the sequence $(w_{n,1})$.

By taking subsequences if necessary, we can assume that $w_{n,k}\to w_{k}$, $k=1,2$, as $n\to\infty$. 
By continuity, $h(w_1)=h(w_2)$, and since $|w_1|=|w_2|=1$ and $h$ is injective on $|w|=1$, 
this can only happen if $w_1=w_2$. Because $V(r',\epsilon')$ is an open set containing $w_1$, 
the points $w_{n,1}$ and $w_{n,2}$ belong to $V(r',\epsilon')$ for all $n$ sufficiently large, 
which in view of the equality $h(w_{n,1})=h(w_{n,2})$ contradicts the 
fact that $h$ is univalent on $V(r',\epsilon')$.

We have just proved the existence of a number $r_0$ in the range $(1+r')/2<r_0<1$ such that 
for all $z\in \cj{U_-(r_0,\epsilon)}$, the function $f_z(w)=h(w)-\varphi(z)$ has exactly 
one zero in  $2r_0-1\leq |w|\leq 1$, which is given by $\phi(z)\in \cj{V_-(r_0,\epsilon)}$. 
In particular, $|\phi(z)|\geq r_0$, and so    
\begin{align}\label{constant-d}
\beta_1:=\min\left\{|h(w)-\varphi(z)|:|w|=2r_0-1,\ z\in \cj{U_-(r_0,\epsilon)}\right\}>0.
\end{align}
Note also that for all $z\in U_-(r_0,\epsilon)$, $\phi(z)$ is a simple zero of $f_z(w)$ 
because $h(w)$ is conformal on $V(r',\epsilon')\supset V_-(r_0,\epsilon)$. 
Thus, $U_-(r_0,\epsilon)\subset D_1$ and $\phi(z)=\phi_1(z)$ for $z\in U_-(r_0,\epsilon)$.

Moreover, just as in the proof of Theorem \ref{theorem:strongasymptotic}, 
we can use the residue theorem to get from   \eqref{secondrepresentationforQ_n} 
that for all $z\in U_-(r_0,\epsilon)$,
\begin{align*}
Q_n(z)=	{} &(n+1)\phi'(z)\phi(z)^n+	\frac{(n+1)\varphi'(z)}{2\pi i}\int_{|w|=2r_0-1}\frac{w^n}{f_z(w)}\,\diff w.
\end{align*}
Combining this with  \eqref{polythirdexpansion}, \eqref{ineq:alpha_nn}, 
\eqref{constant-c} and \eqref{constant-d}, we obtain the inequality
\begin{align}\label{bound-An}\begin{split}
|A_n(z)|\leq{} & \left|\frac{\sqrt{n+1}}{\alpha_{n,n}}-1\right|
+\sum_{j=1}^\infty\frac{|h(n,j)|(n+j+1)|}{n+1}|\phi(z)|^j\\
&+\beta_0\beta_1\tau_0^n\sum_{j=1}^\infty\frac{|h(n,j)|(n+j+1)|}{n+1}(2r_0-1)^{j+1},
\end{split}
\end{align}
which is valid for  $z\in U_-(r_0,\epsilon)$, where $\tau_0=(2r_0-1)/r_0<1$.

On  $U_-(r_0,\epsilon)$, we have $|\phi(z)|\geq r_0$, and 
by \eqref{estimateforgamma_n}, there exists a constant $\beta$ independent of $n$
such that  
\begin{align*}
\left|\frac{\sqrt{n+1}}{\alpha_{n,n}}-1\right|	\leq \frac{\beta_2}{n},\qquad n\geq 1.
\end{align*} 
This and \eqref{estimado2} allow us to deduce from \eqref{bound-An} that for  $z\in U_-(r_0,\epsilon)$ and some constant $\beta_3$,
\begin{align*}
|A_n(z)|\leq{} & \frac{\beta_2}{n}+\frac{E_1}{n(1-|\phi(z)|)}+\frac{E_1}{n^{B_1}(1-|\phi(z)|)^{B_1}}+\frac{\beta_3\tau_0^n}{n}\\
\leq {} & \frac{\beta_2+E_1+\beta_3}{n(1-|\phi(z)|)}+\frac{E_1}{n^{B_1}(1-|\phi(z)|)^{B_1}},
\end{align*}
 whence \eqref{growth-at-boundary}
follows.\end{proof}

\subsection{Proof of Lemma~\ref{maintheorem:strongasymptotic}}
\noindent  As in the previous subsection, we use $z_1,\ldots,z_q$ to denote the corners of $L$.  
We need to show that for every $z_0\in L\setminus\{z_1,\ldots,z_q\}$,
there exists a neighborhood $\caliU_{z_0}$ of $z_0$ 
and a constant $C_{z_0}$ such that $\caliU_{z_0}\subset \Omega^*$  and 
\begin{align}\label{ineq-hn}
\left|A_n(z)\right|\leq  \frac{C_{z_0}\log n}{n},\qquad z\in \caliU_{z_0}.
\end{align}
We use the notation introduced in  \eqref{def:A}, \eqref{def:V}, and \eqref{def:U}. 
Let $\epsilon>0$ be chosen so small that \eqref{angularcondition} is satisfied and 
\[
w_0:=\phi(z_0)\not \in \bigcup_{k=1}^q\{e^{i\theta}:\theta_k-\epsilon\leq\theta\leq\theta_k+\epsilon\}.
\]
By Proposition \ref{prop-asympt-near-boundary}, for such a $\epsilon$
we can find $0<r<1$ such that 
\begin{align}\label{growth-at-boundary-1}
\left|A_n(z)\right|\leq {} &\frac{M_1}{n(1-|\phi(z)|)}+\frac{M_1}{n^{B_1}(1-|\phi(z)|)^{B_1}},
\qquad z\in U_-(r,\epsilon),\quad n\geq 1,
\end{align}
where $M_1$ and $B_1$  are independent of $n$ and $z$. 
By increasing $r$ if necessary, we can guarantee that $\cj{U(r,\epsilon)}\subset \Omega^*$, so that 
\begin{align}\label{ineq:inf-phi-prime}
\inf\{|\phi'(z)|:z\in U(r,\epsilon)\}>0.
\end{align}

The estimate \eqref{second-ineq-A_n} holds on $U_+(r,\epsilon)$, and $|\phi(z)|<1/(1-r)$ for all 
$z\in U_+(r,\epsilon)$. In combination with  \eqref{growth-at-boundary-1}, this yields an inequality of the form
\begin{align}\label{growth-at-boundary-2}
\left|A_n(z)\right|\leq {} &\frac{M_2}{n(1-|\phi(z)|)}+
\frac{M_2}{n^{B_1}(1-|\phi(z)|)^{B_1}},
\end{align}
which is valid for $z\in U_-(r,\epsilon)\cup U_+(r,\epsilon)$ and $n\geq 1$, with $M_2$ some constant.

By \eqref{second-ineq-A_n}, on the level curve 
\[
L_{1/(1-r)}:=\{z\in \C:|\phi(z)|=1/(1-r)\},
\]
we have the bound 
\[
|p_n(z)|\leq  \frac{M_3\sqrt{n+1}}{(1-r)^n},\qquad z\in L_{1/(1-r)},\quad n\geq 0,
\]
where the constant $M_3$ only depends on $r$. By the maximum principle, the same bound holds inside $L_{1/(1-r)}$, and in particular, on the set 
$U(r,\epsilon)$. This and \eqref{ineq:inf-phi-prime} imply that for all $ n\geq 0$,
\begin{align}\label{global-bound}
|A_n(z)|=\left|\frac{p_n(z)}{\sqrt{n+1}\phi'(z)\phi(z)^n}-1\right|\leq \frac{M_4}{(1-r)^{2n}},\qquad z\in U(r,\epsilon),
\end{align}
where $M_4$ is some constant.

For  $\eta>0$, the M\"{o}bius transformation  
\[
J(\zeta)=w_0\frac{i-\eta\zeta}{i+\eta\zeta},
\]
maps the real line onto the unit circle, while it takes the lower-half plane $H_-:= \{\zeta:\I(\zeta)<0\}$ onto 
the unit disk $\D(0,1)$, and the upper half-plane $H_+:= \{\zeta:\I(\zeta)>0\}$ onto $\Delta(0,1)$. 
Since $J(0)=w_0$, by choosing $\eta$ sufficiently 
small we can guarantee that  $J(\Sigma_1)\subset V(r,\epsilon)$, 
where $\Sigma_1$ is the open unit square, that is, the set corresponding to $\delta=1$ in the definition 
\[
\Sigma_\delta:=\left\{\zeta\in\C:|\R(\zeta)|<\delta,\, |\I(\zeta)|< \delta\right\}.
\] 
The set $\psi(J(\Sigma_1))$ is open, contains $z_0$, 
and  we have the inclusions 
\[
\psi(J(\Sigma_1\cap H_+))\subset U_+(r,\epsilon),
\qquad\psi(J(\Sigma_1\cap H_-))\subset U_-(r,\epsilon).
\]
Combining \eqref{growth-at-boundary-2} and \eqref{global-bound} with the equality
\begin{align*}
	|1-|J(\zeta)||=\frac{4\eta|\I(\zeta)|}{|i+\eta \zeta|^2(1+|J(\zeta)|)},
\end{align*}
we find that the functions 
$$u_n(t):=A_n(\psi(J(\zeta))),\qquad \zeta\in \Sigma_1,$$ satisfy the following properties, where $M_5$ is a certain constant independent of the index $n$:
\begin{enumerate}[leftmargin=.9cm]
\item[(i)] $u_n(\zeta)$ is analytic in $\Sigma_1$; 
\item[(ii)] ${\displaystyle |u_n(\zeta)|\leq \frac{M_5}{n|\I(\zeta)|}+\frac{M_5}{n^{B_1}|\I(\zeta)|^{B_1}}  }$,
\quad  $\zeta\in \Sigma_1\setminus (-1,1)$; 
\item[(iii)] ${\displaystyle |u_n(\zeta)|\leq M_5a^{n}}$,\quad$a=(1-r)^{-2}$,\quad $\zeta\in \Sigma_1$.  
\end{enumerate}

If we prove that for every $0<\delta<1$, there is a constant 
$C_{\delta}$ such that
\begin{align}\label{ineq-fn}
\left|u_n(\zeta)\right|= C_{\delta}\frac{\log n}{n},
\qquad \zeta\in \Sigma_\delta,\quad n>1 ,
\end{align} 
then \eqref{ineq-hn} would hold with the choice of neighborhood 
$\caliU_{z_0}=\psi(J(\Sigma_{\delta}))$, $\delta<1$.
The assertion \eqref{ineq-fn} is an immediate consequence of the following proposition.

\begin{proposition}
Let $\{u_n\}_{n\ge 1}$ be a sequence of holomorphic 
functions on $\Sigma_1$, such that for some numbers $1=b_1\leq b_2\leq \cdots\leq b_m$, it holds that
\begin{equation}
\label{eq:conv-away-segment}
|u_n(z)|\leq\sum_{j=1}^m\frac{1}{n^{b_j}|\I(z)|^{b_j}},\qquad z\in \Sigma_1\setminus(-1,1).
\end{equation}
Assume, in addition, that for some real number $a\geq 1$, we have the uniform bound
\begin{equation}
\label{eq:uniform-bound}
|u_n(z)|\leq a^n,\qquad z\in \Sigma_1.
\end{equation}
Then for any $0<\delta<1$, there exists
a constant $C_{\delta}$ such that
\[
\left|u_n(z)\right|\leq\frac{C_{\delta}\log n}{n},
\qquad z\in \Sigma_{\delta},\quad n>1.
\]
In particular, $u_n(z)$ converges to $0$ locally 
uniformly on $\Sigma_1$.
\end{proposition}

\begin{proof}
Fix $0<\delta<1$, and put 
\begin{align}\label{def-mu}
\sigma:=\frac{4(1+\epsilon)}{\pi(1-\delta)}>1,
\end{align}
where $\epsilon>0$ is some arbitrary positive number.
For $n>1$, we 
denote by $S_n$ the thin rectangle
\[
S_n=\left\{z\in \Sigma_1: |\I(z)|\leq\frac{1}{\sigma \log n}\right\}.
\]
For $z\in \Sigma_1\setminus S_n$, we have $|\I(z)|\geq (\sigma \log n)^{-1}$, 
so the bound \eqref{eq:conv-away-segment}
gives 
\begin{align*}
|u_n(z)|\leq  \sum_{j=1}^m \frac{\sigma^j(\log n)^{b_j}}{n^{b_j}}
\leq  m\sigma^{b_m}\frac{\log n}{n},\qquad z\in \Sigma_1\setminus S_n.
\end{align*}
It remains to establish a similar bound on the set 
$\Sigma_{\delta} \cap S_n$.
For this, it is in turn sufficient to prove that 
for any fixed $x_0\in (-\delta,\delta)$,
we have
\[
\big|u_n(x_0+ i  y)\big|
\leq \frac{C_{\delta}\log n}{n},\qquad 
|y|\leq \frac{1}{\sigma\log n},
\]
where $C_{\delta}$ is a constant independent of $x_0$.

The proof will be based on rescaling 
\[
\{z\in S_n\,:\,|\R(z)-x_0|<\delta'\}, \qquad \delta':=1-\delta,
\]
to a rectangle of fixed (unit) height and growing width, 
and then apply a quantitative analogue of the 
Phragmén-Lindelöf principle in a strip.
To that end, we denote by $T_n$ the rescaled rectangle
\begin{align*}
T_n
&:=\left\{ (\sigma\log n) (z-x_0):z\in S_n,\ |\R(z)-x_0|\leq \delta'\right\}\\
&=\left\{\zeta\in\C: |\I(\zeta)|\leq  1,\ |\R(\zeta )|\leq \delta'\sigma \log n\right\},
\end{align*}
and consider the sequence $\{g_n\}_{n\ge 0}$ of holomorphic 
functions on $T_n$ defined by 
\[
g_n(\zeta):=u_n\big(x_0+\zeta(\sigma \log n)^{-1}\big),\qquad \zeta\in T_n.
\] 
On the horizontal parts of the boundary $\partial T_n$,
the bound \eqref{eq:conv-away-segment} implies that
\begin{align}
\label{eq:g-bd-boudnary}
\begin{split}
\left|g_n\left(x\pm  i  \right)\right|
={}&\left|u_n\left(x_0+\frac{x}{\sigma\log n}\pm \frac{i}{\sigma\log n}\right)\right|
\\
\leq{} & m\sigma^{b_m}\frac{\log n}{n},\qquad |x|\leq \delta'\sigma\log n.
\end{split}
\end{align}
In addition, the uniform bound \eqref{eq:uniform-bound} translates to
\begin{equation}
\label{eq:g-bd-boundary-4}
\left|g_n(\zeta)\right|\leq a^n,\qquad \zeta \in T_n.
\end{equation}

We next introduce the auxiliary function
\[
h(\zeta):=\exp\left(-\sqrt{2}\left(e^{\pi \zeta /4}+e^{-\pi \zeta /4}\right)\right).
\]
The function $h$ is entire and $|h(x+ i  y)|$ tends to zero rapidly as $|x|\to\infty$.
In analogy with standard proofs of the Phragmén-Lindelöf principle
in a strip, we will use $h$ to subdue any growth of 
$|g_n(\zeta)|$ in the horizontal direction,
allowing us to apply the maximum principle to $|g_n\,h|$ in $T_n$. 
More specifically, for $\zeta=x+ i  y\in T_n$, we have
\begin{align}
\begin{split}
\label{eq:h-bd}
|h(\zeta)|&=\exp\Big(-\sqrt{2}\cos(\pi y/4)
\big(e^{\pi x/4}+e^{-\pi x/4}\big)\Big)\\
&\leq \exp\Big(-
\exp\big(\pi |x|/4\big)\Big),
\end{split}
\end{align}
where the inequality follows from the fact that $\cos(\pi y/4)\ge \frac{1}{\sqrt{2}}$
for $|y|\leq 1$ together with the elementary bound $e^{t}+e^{-t}\ge e^{|t|}$.

We have $|h(\zeta)|\leq 1$ throughout the strip $\{\zeta:|\I(\zeta)|\leq 1\}$,
so \eqref{eq:g-bd-boudnary} implies that 
\begin{align}
\label{eq:bd-hg-horizontal}
\begin{split}
|h(\zeta)g_n(\zeta)|\leq{} m\sigma^{b_m}\frac{\log n}{n},
\qquad \zeta=x\pm  i  \in \partial T_n.
\end{split}
\end{align}
On the vertical parts of the boundary, \eqref{eq:h-bd} and \eqref{def-mu} instead give
\begin{align*}
\left|h\left(\pm \delta'\sigma \log n+ i  y\right)\right|
\leq {} &\exp\left(-
e^{\tfrac{\pi \delta'\sigma \log n}{4}}\right)\\
\leq {} &\exp\left(-
n^{1+\epsilon}\right),
\qquad |y|\leq 1.
\end{align*}
Combining this with the uniform bound \eqref{eq:g-bd-boundary-4}, 
we see that the product satisfies
\begin{equation}
\label{eq:gh-bd-vert}
\left|(g_n\cdot h)\left(\pm \delta'\sigma \log n+ i  y\right)\right|
\leq a^n\exp\left(-n^{1+\epsilon} \right),\qquad |y|\leq 1
\end{equation}
on the vertical parts of $\partial T_n$.
Since the right-hand side of \eqref{eq:gh-bd-vert} 
tends to $0$ faster than  $(\log n)/n$,
there exists some constant $C_{\delta}>m\sigma^{b_m}$ such that
\[
\left|(g_n\cdot h)\left(\pm \delta'\sigma \log n+ i  y\right)\right|
\leq \frac{C_{\delta}\log n}{n}, \qquad |y|\leq 1,\quad n\geq 1.
\]
When combined with \eqref{eq:bd-hg-horizontal},
\eqref{eq:gh-bd-vert} gives 
\[
\left|g_n(\zeta)h(\zeta)\right|
\leq  \frac{C_{\delta}\log n}{n},\qquad \zeta\in\partial T_n,
\]
which, by the standard maximum principle, also holds for all $ \zeta\in T_n$.
In other words, for $\zeta=x+ i  y\in T_n$, we have 
\begin{equation}
\label{eq:bound-g-final}
|g_n(\zeta)|\leq  \frac{C_{\delta}\log n}{n|h(\zeta)|}
= \frac{C_{\delta}\log n}{n}\exp\Big(\sqrt{2}\cos(\pi y/4)\,
\big(e^{\pi x/4}+e^{-\pi x/4}\big)\Big).
\end{equation}

Recall that we want to analyze the size of $u_n(z)$ on the line $\R (z)=x_0$,
which corresponds to evaluating $g_n(\zeta)$ along the imaginary axis.
Invoking the bound \eqref{eq:bound-g-final} for $\zeta= iy$ and $|y|\leq 1$, we obtain
\begin{align*}
\left|u_n\left(x_0+ i  y(\sigma \log n)^{-1}\right)\right|
={}&\big|g_n( i  y)\big|\leq  \frac{C_{\delta}\log n}{n}
\exp\left(2\sqrt{2}\cos\left(\pi y/4\right)\right)\\
\leq{} & \frac{C_\delta e^{2\sqrt 2}\log n}{n}.
\end{align*}
This completes the proof.
\end{proof}

\subsection{Proof of Theorem~\ref{thm:exteriorPhi}}
\noindent We know from \eqref{second-ineq-A_n} and \eqref{asymptoticformulaindideD} that $|A_n(z)|=\Ordo(1/n)$ 
uniformly as $n\to\infty$ on closed subsets of $\Omega\cup D_1$. Hence,
it is sufficient to show that the bound $A_n(z)=\Ordo(\log n/n)$ holds near any $z_0\in L$
away from the set of corners $C(L)$, which is precisely the assertion of Lemma~\ref{maintheorem:strongasymptotic}.

\subsection*{Acknowledgements}
\noindent The first author expresses his gratitude to Arno Kuijlaars for helping support his research stay at KU Leuven during the Spring of 2024. This stay was also partially supported by a 
grant from The University of Mississippi Office of Research and Sponsored Programs.\\

\noindent The research of the second author was supported by
the Flemish Research Foundation (FWO) 
Odysseus Grant (No.\ G0DDD23N).

\bigskip
\medskip

\smallskip\noindent{Erwin Mi\~{n}a-D\'{\i}az\newline
Department of Mathematics \newline
The University of Mississippi, \newline 
Hume Hall 305, P.~O.~Box 1848,\newline
MS 38677-1848, USA
\newline {\tt minadiaz@olemiss.edu}

\smallskip
\bigskip

\smallskip\noindent{Aron Wennman\newline
Department of Mathematics \newline
KU Leuven, \newline 
Celestijnenlaan 200B,\newline
3001 Leuven, Belgium
\newline {\tt aron.wennman@kuleuven.be}

\begin{thebibliography}{99}
\bibitem{BeckermannStylianopoulos}{B. Beckermann, N. Stylianopoulos,
Bergman orthogonal polynomials and the Grunsky matrix, Constr. Approx. 47 (2018), 211-235.}
	
\bibitem{Carleman}{T. Carleman, \"{U}ber die approximation 
analytischer funktionen durch lineare aggregate von vorgegebenen
potenzen,  Archiv. f\"{o}r Math. Atron. och Fysik, 17
(1922) 1-30.} 
	
\bibitem{Davis}{P.J. Davis, The Schwarz Function and Its Applications, The Carus Mathematical Monographs 17, 1974.}	
	
\bibitem{DragnevMinaDiaz1}{P. Dragnev, E. Mi\~na-D\'iaz, Asymptotic behavior and zero distribution of 
Carleman orthogonal polynomials, J. Approx. Theory 162 (2010) 1982-2003.}

\bibitem{DragnevMinaDiaz2}{P. Dragnev, E. Mi\~na-D\'iaz, On a series representation for 
Carleman orthogonal polynomials, Proc. Amer. Math. Soc. 138 (2010) 4271-4279.}
	
\bibitem{DragnevMinaNorthington}{P. Dragnev, E. Mi\~na-D\'iaz, M. Northington V, 
Asymptotics of Carleman polynomials for level curves of the inverse of a shifted Zhukovsky transformation, 
Comput. Methods Funct. Theory 13 (2013) 75-89.}
	
\bibitem{ES}{M. Eiermann, H. Stahl, 
Zeros of orthogonal polynomials on regular $N$-gons, 
Lecture Notes in Mathematics, Vol. 1574, Springer, Heidelberg, 1994, pp. 187–189.}
	
\bibitem{LevinSS}{A.L. Levin, E.B. Saff, N. Stylianopoulos, 
Zero distribution of Bergman orthogonal polynomials for certain planar domains, 
Constr. Approx. 19 (2003) 411-35.}

\bibitem{MaymeskulS}{V. Maymeskul, E.B. Saff, 
Zeros of polynomials orthogonal over regular $N$-gons, 
J. Approx. Theory 122 (2003) 129-40.}
	
\bibitem{MinaDiaz2}{E. Mi\~na-D\'iaz, 
An asymptotic integral representation for Carleman orthogonal polynomials, 
Int. Math. Res. Notices (2008) article ID rnn065.}
	
\bibitem{MinaDiaz1}{E. Mi\~na-D\'iaz, 
Orthogonal polynomials in weighted Bergman spaces, J. Approx. Theory, Article 105972, Vol. 296 (2023).}

\bibitem{MinaDiaz3}{E. Mi\~na-D\'iaz, On the asymptotic behavior of Faber polynomials for domains with piecewise analytic boundary, Constr. Approx. 29 (2009) 421–448.}

	
\bibitem{PS}{N. Papamichael, E. B. Saff, J. Gong, 
Asymptotic behavior of zeros of Bieberbach polynomials, 
J. Comput. Appl. Math. 34 (1991) 325-342.}
	
\bibitem{SaffStylianopoulos}{E.B. Saff, N. Stylianopoulos, 
On the zeros of asymptotically extremal polynomial sequences in the plane, 
J. Approx. Theory 191 (2015) 118-127.}
	
\bibitem{Nikos1}{N. Stylianopoulos, 
Strong asymptotics for Bergman polynomials over domains with corners and applications, 
Constr. Approx. 38 (2013) 59-100}.

\bibitem{Suetin}{P.K. Suetin, 
Polynomials Orthogonal over a Region and Bieberbach Polynomials, 
Proceedings of the Steklov Institute of Mathematics, Providence, RI: Amer. Math. Soc., 1974. }
	
\bibitem{Walsh}{J.L. Walsh, 
Interpolation and Approximation by Rational functions in the Complex Domain, 
Amer. Math. Soc. Colloq. Publ. 20, Amer. Math. Soc., Providence, RI, 5th ed., 1969.}
\end{thebibliography}
\end{document}